\documentclass[10pt,leqno, english]{amsart}

\usepackage[utf8]{inputenc}
\usepackage{amsmath}
\usepackage{amssymb}
\usepackage{amsfonts}
\usepackage{amsthm}
\usepackage{verbatim}
\usepackage[pdftex]{hyperref}
\usepackage{geometry}
\usepackage{hyperref}
\usepackage{pdfsync}
\usepackage{calc}
\usepackage{enumerate,amssymb}
\usepackage[arrow,curve,matrix,tips,2cell]{xy}
\SelectTips{eu}{10} \UseTips
\UseAllTwocells
\usepackage{graphicx}
\usepackage{pb-diagram}
\usepackage{pdfsync}
\usepackage{mathtools}
\usepackage{amscd}
\usepackage[mathscr]{euscript}
\usepackage{tikz}
\usepackage{tikz-cd}
\usetikzlibrary{cd}
\usetikzlibrary{matrix}
\usetikzlibrary{decorations.markings}
\usepackage{cleveref}
\usepackage{upgreek}
\usepackage{arydshln}
\usepackage{listings}

\DeclareMathAlphabet\EuR{U}{eur}{m}{n}
\SetMathAlphabet\EuR{bold}{U}{eur}{b}{n}

\makeindex  

\theoremstyle{plain}
\newtheorem{theorem}{Theorem}[section]
\newtheorem{lemma}[theorem]{Lemma}
\newtheorem{proposition}[theorem]{Proposition}

\newtheorem{conjecture}[theorem]{Conjecture}

\theoremstyle{definition}
\newtheorem{definition}[theorem]{Definition}
\newtheorem{example}[theorem]{Example}

\newtheorem{remark}[theorem]{Remark}

{\catcode`@=11\global\let\c@equation=\c@theorem}

\newcommand{\GL}{\mathrm{GL}}
\newcommand{\SL}{\mathrm{SL}}
\newcommand{\PGL}{\mathrm{PGL}}
\newcommand{\PSL}{\mathrm{PSL}}

\newcommand{\mor}{\mathrm{mor}}
\newcommand{\Maps}{\mathrm{Maps}}
\newcommand{\Hom}{\mathrm{Hom}}

\newcommand{\uE}{\underline{E}}

\newcommand{\C}{\mathbb{C}}

\newcommand{\Q}{\mathbb{Q}}
\newcommand{\Z}{\mathbb{Z}}
\newcommand{\N}{\mathbb{N}}

\newcommand{\Ab}{\mathfrak{Ab}}
\newcommand{\fF}{\mathfrak{F}}

\newcommand{\bfb}{\mathbf{b}}

\newcommand{\bfd}{\mathbf{d}}

\newcommand{\cA}{\mathcal{A}}
\newcommand{\cH}{\mathcal{H}}
\newcommand{\cO}{\mathcal{O}}
\newcommand{\XX}{\mathcal{X}}
\newcommand{\YY}{\mathcal{Y}}

\newcommand{\ev}{\operatorname{ev}}
\newcommand{\pt}{\{\bullet\}}

\newcommand{\colim}{\operatorname{colim}}
\newcommand{\cone}{\operatorname{cone}}
\newcommand{\topo}{\operatorname{top}}

\newcommand{\ind}{\operatorname{ind}}
\newcommand{\res}{\operatorname{res}}
\newcommand{\cores}{\operatorname{cores}}

\newcommand{\xycomsquare}[8]                   
{\xymatrix
	{#1 \ar[r]^{#2} \ar[d]^{#4} &
		#3 \ar[d]^{#5}  \\
		#6\ar[r]^{#7} &
		#8
	}
}

\newcommand{\curs}{\EuR}

\newcommand{\GROUPOIDS}{\curs{GROUPOIDS}}

\newcommand{\Or}{\curs{Or}}

\newcommand{\SPACES}{\curs{SPACES}}

\newcommand{\SPECTRA}{\curs{SPECTRA}}

\title[Hecke operators on Bianchi groups]{Hecke operators in Bredon (co)homology, K-(co)homology and Bianchi groups}

\author{David Muñoz}
\address{Department of Computational Mathematics, Science and Engineering, Michigan State University, East Lansing, MI, United States}

\email{munozram@msu.edu}

\author{Jorge Plazas}
\address{Departamento de Matemáticas,  Pontificia Universidad Javeriana, Bogotá D.C, Colombia}

\email{jorge.plazas@javeriana.edu.co}

\author{Mario Vel\'asquez}
\address{Bogot\'a D.C, Colombia}

\email{mavelasquezm@gmail.com}
\urladdr{https://sites.google.com/site/mavelasquezm/}

\date{\today}



\begin{document}

\maketitle

\begin{abstract}
	
	In this article we provide a framework for the study of Hecke operators acting on the Bredon (co)homology of an arithmetic discrete group. Our main interest lies in the study of Hecke operators for Bianchi groups. Using the Baum-Connes conjecture, we can transfer computations in Bredon homology to obtain a Hecke action on the $K$-theory of the reduced $C^{*}$-algebra of the group. We show the power of this method giving explicit computations for the group $\SL_2(\Z[i])$. In order to carry out these computations we use an Atiyah-Segal type spectral sequence together with the Bredon homology of the classifying space for proper actions.
\end{abstract}

\section{Introduction}

Hecke operators play a prominent role in the study of arithmetic groups. The action of Hecke operators on various cohomology theories associated to arithmetic groups and their symmetric spaces provides an essential tool bridging the analytic and arithmetic aspects of the theory. An important class of arithmetic groups that has received a lot of attention recently is that of Bianchi groups. These are groups of the form $\PSL_2(\mathcal{O}_{\Q (\sqrt{-D} )})$ where $\mathcal{O}_{\Q (\sqrt{-D} )}$ is the ring of integers of an imaginary quadratic field. Bianchi groups were first studied by Luigi Bianchi and others in the 1890's as a natural extension of the study of the Modular group $\PSL_2(\Z)$. Bianchi studied in \cite{Bianchi1892} their algebraic properties, finding generators for many of these groups and showing that each Bianchi group acts discontinuously on the hyperbolic 3-space. Bianchi also  developed the tools from reduction theory for binary Hermitian forms needed in the study of this class of groups (cf. \cite{Swan}).

For free subgroups of Bianchi groups, Mesland and \c{S}eng\"{u}n in \cite{MeslandSengun20} have recently defined a Hecke action on $K$-homology using Kasparov's bivariant $KK$-theory. We tackle the general case by first defining Hecke operators on Bredon (co)homology, allowing us to then transfer the computations to $K$-theory in full generality. Computations in Bredon (co)homology can be carried out using  Atiyah-Segal type spectral sequences. We develop the corresponding machinery which we then apply to explicitly compute the Hecke action on the $K$-homology of the group $\PSL_2(\Z[i])$.

The plan of the article is as follows. In section 2 we review the definition of Bredon modules and Bredon (co)homology. Given a discrete group $G$, Bredon modules associated to families of subgroups of $G$ provide coefficient systems for $G$-equivariant (co)homology theories. We focus in the case where the coefficient system is given by the representation ring. In this case the Bredon (co)homology of a $G$-CW-complex $X$ is given in terms of the representation rings of its cell stabilizers. In section 3 we define equivariant K-(co)homology in terms of spectra and discuss its relation to Bredon (co)homology using spectral sequences. This in turn leads via the Baum-Connes conjecture to a description of the $K$-theory of the reduced $C^{*}$-algebra of $G$ and the possibility of defining Hecke operators at the level of such algebras. In section 4 we review the theory of Hecke algebras and introduce the natural Hecke action on group cohomology. Here we also discuss Hecke correspondences over a $G$-space. In section 5 we develop the machinery necessary to define Hecke operators in Bredon (co)homology and transfer these to equivariant $K$-(co)homology. The core of our treatment lies in identifying the appropriate restriction, corestriction and conjugation morphisms necessary leading to the action of Hecke correspondences. Section 6 of the article is devoted to computations for Bianchi groups. Expressing Bianchi groups as amalgamated products, we can carry out the computations in terms of the representation theory of their factors viewed as cell stabilizers on $G$-spaces. Explicit computations in the case of Hecke operators associated to congruence subgroups of $\PSL_2(\Z[i])$ of prime level are carried out in full. We close the article in section 7 with a few concluding remarks.

\section{Bredon (Co)homology}
\subsection{Bredon (Co)homology}

Bredon (co)homology for finite groups was introduced by Bredon in \cite{Bredon67-1}, \cite{Bredon67-2} in order to provide an appropriate framework for coefficient systems in an equivariant (co)homology theory. The theory can be extended to arbitrary topological groups (cf.\cite{Illman75}). In this section we recall the main aspects of the theory for discrete groups. Our treatment follows that of \cite{Sanchez05} and  \cite{MislinValette03}. 

In order to define Hecke operators we will need to develop in full generality the theory underlying Bredon's definition of cohomology with coefficients. Given a discrete group $G$, its Bredon cohomology groups with coefficients in the representation ring will be computed via a cochain complex where each term is the representation ring of the stabilizer of an $n$-cell of the classifying space for proper actions of $G$ viewed as a $G-CW$-complex. We will develop the necessary machinery leading to this description in this section and the next one. This section will be devoted to the description of the cohomology with coefficients for a  $G-CW$-complex. In the next section we recast the discussion in terms of spectra leading to the possibility of carrying out the computations using the classifying space for proper $G$-actions.

We begin by defining the orbit category which is central to  Bredon's definition of (co)homological invariants for spaces  with  a group action. 

\begin{definition}Let $G$ be a discrete group and let $\fF$ be a family of subgroups of $G$, closed under conjugation and finite intersections. Define the {\emph{orbit category} }  $\Or_{\fF}(G)$ as the category whose objects are sets of the form $G/H$ with  $H\in \fF$, and whose  morphisms are given by $G$-maps. Notice that such a morphism  $f_{g}: G/H\rightarrow G/K$ is determined by an element $gK\in G/K$ with $g^{-1}H g\subset K$, so that it sends the coset $H$ to the coset $gK$, i.e. we have an identification
\begin{eqnarray*}
\mor_{\Or_{\fF}(G)} (G/H,G/K) &=&  \Maps(G/H,G/K)^G.
\end{eqnarray*}
When $\fF$ is the family of all subgroups of $G$ we simply denote $\Or_{\fF}(G)$ by $\Or_{}(G)$.
\end{definition}

Throughout what follows we fix a choice of a family $\fF$ of subgroups of $G$ as above. 

Denote by $\Ab$ the category of abelian groups.

\begin{definition}A covariant (resp. contravariant) Bredon module is a covariant (resp. contravariant) functor
\begin{eqnarray*}
M :	\Or_{\fF}(G) & \longrightarrow &  \Ab.
\end{eqnarray*}
A morphism 
\begin{eqnarray*}
	\Psi : M  & \longrightarrow &  N
\end{eqnarray*}
between Bredon modules is given by a natural transformation between the corresponding functors. This means that for each $H\in\mathfrak{F}$ there is a morphism of abelian groups
\begin{eqnarray*}
	\Psi(G/H) :  M(G/H)  & \longrightarrow &  N(G/H)
\end{eqnarray*}
and for every $f_g : G/H \to G/K$ we have, in the covariant case, a commutative diagram
$$\xymatrix{M(G/H)\ar[r]^{M(f_g)}\ar[d]_{\Psi(G/H)}& M(G/K)\ar[d]^{\Psi(G/K)}\\N(G/H)\ar[r]^{N(f_g)}&N(G/K)}$$
whilst in the cotravariant case the horizontal arrows are reversed. 
\end{definition}

If $M$ and $N$ are both covariant (resp. contravariant) Bredon modules, the group structure in each of the $\Hom(M(G/H), N(G/H))$ induces an abelian group structure in the set of na\-tu\-ral transformations $\mor(M, N)$. It can be shown moreover that the category of covariant (resp. contravariant) Bredon modules is an Abelian category.

If $M$ is a contravariant Bredon module and $N$ is a covariant Bredon module, we define the abelian group 
\begin{eqnarray*}
	M\otimes_\mathfrak{F} N  & =  &  \bigoplus_{H\in\mathfrak{F}} M(G/H)\otimes_{\Z} N(G/H) \;\bigg/ \sim
\end{eqnarray*}
where the relation $\sim$ is generated by $M(f)(m)\otimes n-m\otimes N(f)(n)$ for each $f:G/H\rightarrow G/K$, $m\in M(G/K)$, and $n\in N(G/H)$.

Given a CW-complex $Z$ we denote by $C_{\ast}(Z)$ its cellular chain complex. Let $X$ be $G$-CW-complex, for each $n \geq 0$ we can define a contravariant Bredon module $\underline{C_n(X)}$ by 
\begin{eqnarray*}
\underline{C_n(X)} : G/H & \longmapsto & C_n(X^H),
\end{eqnarray*}
where $X^H$ is the subspace of $X$ fixed by the subgroup $H$. Let $\{ \delta_{\alpha}\}$ be the set of $n$-cells of $X$, then there is an isomorphism
\begin{eqnarray*}
	C_n(X^H) & \cong & \bigoplus_{\alpha}\Z[ \delta_{\alpha}^H],
\end{eqnarray*}
where $ \delta_{\alpha}^H$ is the $H$-fixed point set of $\delta_\alpha$; explicitly, $\delta_{\alpha}^H$ is $\delta_\alpha$ if the cell is fixed by $H$, and otherwise is empty, in which case it does not count in the sum. For a morphism $f_{g}:G/H\rightarrow G/K$ we have 
$$\underline{f_{g}}:=\underline{C_n{(X)}}(f_{g}):C_n(X^K)\rightarrow C_n(X^H), \quad  \delta_{\alpha}^K\longmapsto g\cdot  \delta_{\alpha}^K=: \delta_{\alpha_g}^H.$$

For each $H\in\mathfrak{F}$  the usual boundary map $\partial:C_n(X^H)\rightarrow C_{n-1}(X^H)$ induces a boundary map 
\begin{eqnarray*}
\partial:\underline{C_n(X)} & \longrightarrow & \underline{C_{n-1}(X)}.
\end{eqnarray*}

If $M$ is a contravariant Bredon module and $X$ is a  $G$-CW-complex we obtain a cochain complex
$$\text{mor}(\underline{C_{\ast}(X)},M).$$ 
\begin{definition}
Let $X$ be a $G$-CW-complex and let $M$ be a contravariant Bredon module. We define the $n$-th Bredon cohomology group of $X$ with coefficients in $M$ as
\begin{eqnarray*}
\mathcal{H}_{G}^{n}(X;M) & = & H^n(\mbox{mor}(\underline{C_{\ast}(X)},M)).
\end{eqnarray*}
\end{definition}

Analogously, if $N$ is a covariant Bredon module and $X$ is a  $G$-CW-complex we obtain a chain complex
$$ \underline{C_{\ast}(X)}\otimes_\mathfrak{F} N . $$
\begin{definition}
Let $X$ be a $G$-CW-complex and let $N$ be a covariant Bredon module. We define the $n$-th Bredon homology group of $X$ with coefficients in $N$ as
\begin{eqnarray*}
\mathcal{H}_{n}^{G}(X;N) & = &H_n(\underline{C_{\ast}(X)}\otimes_\mathfrak{F} N).
\end{eqnarray*}
\end{definition}

As mentioned above contravariant Bredon modules form an Abelian category. We will now define a class of projective objects which will play an important role in computations. 

Let $K\in\mathfrak{F}$. We define the standard projective contravariant Bredon module $P_K$ as the functor given in objects of $\Or_{\fF}(G)$ by
\begin{eqnarray*}
P_K (G/H) & = & \Z [\mbox{mor}(G/H,G/K)],\qquad \mbox{for }H\in\mathfrak{F},
\end{eqnarray*}
and which associated to a morphism $f:G/H_1\rightarrow G/H_2$, the morphism $P_K(f):P_K(G/H_2)\rightarrow P_K(G/H_1)$ is given by the linear extension of pre-composing with $f$.

For the Bredon modules $P_K$, an appropriate form of the Yoneda Lemma  shows that given a contravariant Bredon module $M$ there is an induced isomorphism of Abelian groups
$$\ev_K:\mbox{mor}(P_K,M) \longrightarrow M(G/K), \qquad \varphi\longmapsto \ev_K(\varphi)=\varphi(G/K)(1).$$

In a similar manner, if $N$ is a covariant Bredon module, there is an isomorphism 
$$P_K\otimes_{\mathfrak{F}}N\cong N(G/K).$$ 
See \cite{MislinValette03} for more information on these isomorphisms.

Let $X$ be a $G$-CW-complex and, as above, let $\{\delta_{\alpha}\}$ be the set of $n$-cells of $X$, and let $\{e_{\beta}\}$ be a set of $G$-representatives of those $n$-cells; we know that 
$$C_n(X^H) \; \cong \;  \bigoplus_{\alpha}\Z[\delta_{\alpha}^H] \;  \cong \;  \bigoplus_{\beta}\Z[(G\cdot e_{\beta})^H].$$
Moreover, if $S_{\beta}$ is the stabilizer of the cell $e_{\beta}$, and the $g$'s are taken as representatives in $G/S_{\beta}$ then there is a $g e_{\beta}$ fixed by $H$ if and only if $g$ is such that $g^{-1}Hg\subset S_{\beta}$, so we have a bijective correspondence 
\begin{eqnarray*}
(G\cdot e_{\beta})^H & = & \mbox{mor}(G/H,G/S_{\beta}).
\end{eqnarray*}
Therefore, we obtain 
$$C_n(X^H) \; \cong \;  \bigoplus_{\beta}\Z[\mbox{mor}(G/H,G/S_{\beta})] \; = \;  \bigoplus_{\beta} P_{S_{\beta}}(G/H),$$
so, as Bredon modules,
\begin{eqnarray*}
	\underline{C_{n}(X)} & \cong & \bigoplus_{\beta}P_{S_{\beta}}.
\end{eqnarray*}

We have an isomorphism of chain complexes 
$$\mbox{mor}_{G}(\underline{C_{\ast}(X)},M)  \, \cong \,  \prod_{\beta^\ast} \mbox{mor}(P_{S_{\beta^\ast}},M) \, \cong \,  \prod_{\beta^\ast} M(G/S_{\beta^\ast}),$$
where $\{\beta^\ast\}$ indexes the $G$-representatives of $\ast$-cells. This becomes a direct sum assuming there are finite representatives for the cells.

\subsection{Coefficients in the representation ring}

Consider now the family $\fF$ of finite subgroups of $G$. For computations of equivariant K-theory and K-homology we will use the contravariant Bredon module $\mathcal{R}$ which acts on objects of $\Or_{\fF}(G)$ by  sending $G/H$ to $R(H)$, the rep\-re\-sen\-ta\-tion ring of the subgroup $H$. At the level of morphisms,  $\mathcal{R}$ acts via the composition of restriction and the isomorphism given by conjugation, so for any $f_{g}:G/H\rightarrow G/L$ the morphism $\mathcal{R}(f_{g})$ is the composition 
$$R(L)\xrightarrow{\;\mbox{Res}_{g^{-1}Hg}^L\;}R(g^{-1}Hg)\xrightarrow{\;\cong\;}R(H).$$
Then, as above, we have an isomorphism describing the Bredon cochain complex
\begin{eqnarray}\label{bredonc}
\mbox{mor}_{G}(\underline{C_{n}(X)},\mathcal{R}) & \cong &\bigoplus_{\alpha}R(S_{\alpha}),
\end{eqnarray}
with the assumption that there are finite orbit representatives of $n$-cells. Here, the coboundary map is given by restriction of representations, from the stabilizer of an $n$-cell to the stabilizer of the corresponding $(n+1)$-cell that contains it.

Similarly, we can consider $\mathcal{R}$ as a covariant Bredon module, setting $\mathcal{R}(f_g)$ to be the composition 
$$R(H)\xrightarrow{\;\cong\;}R(g^{-1}Hg)\xrightarrow{\;\mbox{Ind}_{g^{-1}Hg}^L\;}R(L).$$
Then the chain complex $\underline{C_{n}(X)}\otimes_{\mathfrak{F}}\mathcal{R} $ can be described as 
\begin{eqnarray}
\underline{C_{n}(X)}\otimes_{\mathfrak{F}}\mathcal{R} & \cong & \bigoplus_{\alpha}R(S_{\alpha}),
\end{eqnarray} 
where the boundary map is given by induction of representations.

\begin{remark}
	Once we extend scalars to $\C$ we obtain an isomorphism
	\begin{eqnarray*}
		\left[ \; \bigoplus_{\alpha}R(S_{\alpha}) \,  \right] \otimes_{\Z} \C & \cong & \bigoplus_{\alpha} \mathfrak{Cl}_{c}(S_{\alpha})
	\end{eqnarray*}
	where $\mathfrak{Cl}_{c}(S_{\alpha})$ is the vector space of locally constant complex valued class functions on $S_{\alpha}$, and the individual isomorphisms $R(S_{\alpha}) \otimes_{\Z} \C \cong \mathfrak{Cl}_{c}(S_{\alpha})$ are given by passing from representations to their characters. The spaces
	\begin{eqnarray*}
		\mathrm{C}_{\; n}^{\, \mathrm{ch}}(G ; X)  & = &\bigoplus_{\alpha} \mathfrak{Cl}_{c}(S_{\alpha})
	\end{eqnarray*}
	form a chain complex whose boundary maps are similarly defined by induction of representations. This chain complex computes the {\emph{chamber homology groups}} $\mathrm{H}_{\; n}^{\mathrm{ch}}(G ; X)$. Chamber homology is an equivariant homology theory which, as Bredon homology, incorporates the structure coming from the representation theory of subgroups of $G$. We refer the reader to \cite{BaumEtAl2000} for a treatment of chamber homology together with its relevance in the context of the Baum-Connes conjecture for $p$-adic groups.
	
\end{remark}

\section{Equivariant K-homology and the Baum-Connes conjecture}\label{K-homology}
Now we will describe a topological version of the Baum-Connes assembly map, from this approach we easily obtain the naturality of the assembly map. It will be necessary for the computation in Section \ref{comp}. The proofs of all results in this section can be found in \cite{Lueck19}.

\subsection{Spectra and homology theories}
\begin{definition}
A spectrum $${\bf E} = \{(E(n), \sigma(n)) \mid n \in \Z\}$$ is a sequence of pointed spaces
$$\{E(n)\mid n \in \Z\}$$ together with pointed maps, called structure maps, 
\begin{eqnarray*}
\sigma(n):E(n)\wedge S^1 & \longrightarrow & E(n + 1).
\end{eqnarray*}
A map of spectra $f:{\bf E}\to {\bf E'}$
is a sequence of maps $f(n):E(n)\to E'(n)$
which are compatible with the structure maps $\sigma(n)$. The homotopy groups of a spectrum are defined by
	\begin{eqnarray*}
	\pi_n({\bf E}) & := & \colim_{k\to\infty} \pi_{n+k}(E(k)).
	\end{eqnarray*}
Where the $n$-th structure map of the system $\pi_{n+k}(E(k))$ is given by the composite
$$\pi_{n+k}(E(k))\xrightarrow{S}\pi_{n+k+1}(E(k)\wedge S^1)\xrightarrow{\sigma(k)_*}\pi_{n+k+1}(E(k+1)),$$
where $S$ denotes the suspension homomorphism.
\end{definition}

We will denote by $\SPECTRA$ the category of spectra. 

If ${\bf E}$ is a spectrum, one obtains a (non-equivariant) homology theory $H_*(-; {\bf E})$
by defining for a CW-pair $(X,A)$ and $n\in \Z$ 
\begin{eqnarray*}
H_n(X,A; {\bf E}) & = & \pi_n (X_+ \cup_{A_+}\cone(A_+) \wedge {\bf E}),
\end{eqnarray*}
where $X_+$ is obtained from $X$ by adding a disjoint
base point and $\cone$ denotes the (reduced) mapping cone. The main property of this homology theory is given by the equality $H_n(\pt; {\bf E}) = \pi_n({\bf E})$. The definitions can be extended to an equivariant context using the orbit category.

\subsection{$\Or(G)$-spaces}

Denote by $\SPACES$ the category of topological spaces. Also, recall that if $\fF$ is the family of all subgroups of $G$ we denote $\Or_{\fF}$ by $\Or(G)$. In order to ease notation we will at times use lower case letters to denote objects in $\Or(G)$. 
A covariant $\Or(G)$-space is a covariant functor 
\begin{eqnarray*}
\mathcal{X} : \Or(G) &\longrightarrow &\SPACES.
\end{eqnarray*}

Given  a $G$-space $X$, the {\it fixed point set system} of $X$, denoted by $\Phi X$, is
the $\Or(G)$-space defined by
\begin{eqnarray*}
\Phi X (G/H) & := & {\rm{Maps}}(G/H, X)^G \; = \; X^H
\end{eqnarray*}
and if $\theta: G/H \to G/K$ is a $G$-map corresponding to $gK \in (G/K)^H$ then when $x\in X^K$
\begin{eqnarray*}
\Phi X (\theta)(x) &:=& gx \in X^H.
\end{eqnarray*}
$\Phi$ defines a contravariant functor from
the category of proper $G$-spaces to the category of
$\Or(G)$-spaces.

Let $\XX$  and $\YY$ be $\Or(G)$-spaces, we define the space
\begin{eqnarray}
\XX\times_{\Or(G)}\YY & := & \bigsqcup_{c \in {\rm{Obj}}(\Or(G))} \XX(c) \times \YY(c) / \sim
\end{eqnarray}
where $\sim$ is the equivalence relation generated by $(\XX(\phi)(x),y) \sim (x, \YY(\phi)(y))$
for all morphisms $\phi: c \to d$ in $\Or(G)$ and points $x \in \XX(d)$ and $y \in \YY(c)$.

\begin{definition}
A covariant $\Or(G)$-spectrum is a covariant functor
\begin{eqnarray*}
{\bf E}^G : \Or(G) & \longrightarrow & \SPECTRA
\end{eqnarray*}	
\end{definition} 
If ${\bf E}^G$ is a covariant $\Or(G)$-spectrum and $\YY$ is a $\Or(G)$-space, then one obtains a spectrum $Y\times_{\Or(G)}{\bf E}^G$. Hence, we can extend ${\bf E}^G$ to a covariant functor
\begin{align*}{\bf E}^G_\%:G\text{-CW}^2&\to \SPECTRA\\ (X, A)&\mapsto \Phi(X_+\cup_{A_+} \cone(A_+))\times_{\Or(G)}{\bf E}^G.\end{align*}

\begin{lemma}
If ${\bf E}^G$ is a covariant $\Or(G)$-spectrum, then we obtain a $G$-homology theory $H^G_*(-;{\bf E}^G)$ defined as
\begin{eqnarray*}
H^G_n(X,A;{\bf E}^G) & = & \pi_n({\bf E}^G_\%(X,A))
\end{eqnarray*}	
satisfying $H^G_n(G/H;{\bf E}^G) = \pi_n({\bf E}^G(G/H))$ for $n\in \Z$ and $H\subseteq G$.
\end{lemma}

\begin{proof}
	See Lemmas 2.3 and 2.5 in \cite{Lueck19}.
\end{proof}

Recall that the reduced group $C^{*}$-algebra $C_r^*(G)$ is the norm closure of the complex group ring $\C G$ embedded into the space $\mathcal{B}(L^2(G))$ of bounded operators $L^2(G)\to L^2(G)$ equipped with the sup norm given by the right regular representation. Denote by $\GROUPOIDS$ the category of groupoids. There is a covariant functor,
respecting equivalences,
\begin{eqnarray*}
{\bf K}^{\topo} : \GROUPOIDS \longrightarrow \SPECTRA,
\end{eqnarray*}
such that for every group $G$ and all $n\in \Z$ we have
\begin{eqnarray*}
\pi_n({\bf K}^{\topo}(G)) & \cong & K_n(C_r^*(G)),
\end{eqnarray*}
where $K_n(C_r^*(G))$ is the topological K-theory of the reduced group $C^{*}$-algebra $C_r^*(G)$, see \cite{Blackadar06}.

\begin{conjecture}[Baum-Connes Conjecture] A group $G$ satisfies the Baum-Connes Conjecture if the assembly map induced by the projection $pr:\underbar{E}G\to G/G$ 
\begin{eqnarray*}
H^G_n(pr;{\bf K}^{\topo}) :  H^G_n(\underbar{E}G; {\bf K}^{\topo}) & \longrightarrow & H^G_n(G/G;{\bf K}^{\topo}) = K_n(C^*_r(G)),
\end{eqnarray*}
is bijective for all $n\in\Z$.
\end{conjecture}
The homology theory $H_*^G(-;{\bf K}^{\topo})$ is called \emph{$G$-equivariant K-homology} and is also denoted by $K_n^G(-)$. The map $H^G_n(pr;{\bf K}^{\topo})$ is called the \emph{assembly map} and is also denoted by $\mu_{G,n}$.

\subsection{K-theory}

The equivariant cohomology theory associated to the $\Or(G)$-spectra ${\bf K}^{\topo}$ is called \emph{equivariant K-theory} and is denoted by $K_G^n(-)$. By results of L\"{u}ck and Oliver in \cite{LueckOliver01-1} we know that in the case of proper $G$-actions on $G$-CW-complexes, $K_G^*(-)$ can be defined in terms of vector bundles. We summarize the construction in the following lines.

\begin{definition}
For any discrete group $G$ and any finite proper $G$-CW-complex $X$, let $\mathbb{K}_G(X)=\mathbb{K}_G^0(X)$ be the Grothendieck group of the semigroup $\text{Vect}_G(X)$ of isomorphism classes of $G$-vector bundles over $X$ together with direct sum. Define $\mathbb{K}_G^{-n}(X)$, for all $n>0$, by setting 
$$\mathbb{K}_G^{-n}(X) = \ker(\, \mathbb{K}_G(X\times \mathbb{S}^n) \xrightarrow{\;\;\emph{incl}^\ast\;\;} \mathbb{K}_G(X) \,).$$
For any proper $G$-CW-pair $(X,A)$, and $n\geq0$, set 
$$\mathbb{K}_G^{-n}(X,A) = \ker(\, \mathbb{K}_G^{-n}(X\cup_A X) \xrightarrow{\;\;i_2^\ast\;\;} \mathbb{K}_G^{-n}(X) \,).$$
And, let $\mathbb{K}_G^n(X)=\mathbb{ K}_G^{-n}(X)$ and $\mathbb{ K}_G^n(X,A)=\mathbb{ K}_G^{-n}(X,A)$.
\end{definition}
We have then the following equivalence of equivariant homology theories.
\begin{theorem}
Let $G$ be a discrete group, $(X,A)$ be a proper $G$-CW-pair and $n$ be an integer number. There is a natural isomorphism
\begin{eqnarray*}
K_G^n(X,A) & \cong & \mathbb{ K}_G^n(X,A).
\end{eqnarray*}
\end{theorem}
\begin{proof}See \cite{LueckOliver01-1}.
\end{proof}

As in the classical case,  there is an Atiyah-Hirzebruch type spectral sequence converging to equivariant K-theory (respectively $K$-homology) whose $E_2$-term is the Bredon cohomology (respectively Bredon homology) with coefficients in the Bredon module of complex representations $\mathcal{R}$ (see \cite{LueckOliver01-1} for details). 

More concretely, given a discrete group $G$ and any finite dimensional proper $G$-complex $X$, the equivariant skeletal filtration of $X$ induces a cohomological spectral sequence with
\begin{eqnarray*}
E^{p,2q}_{2}(X) \; \cong \; \mathcal{H}^{p}_{G}(X;\mathcal{R}) & \Longrightarrow & K^{\ast}_{G}(X),
\end{eqnarray*}
and a homological spectral sequence with
\begin{eqnarray*}
E_{p,2q}^{2}(X) \; \cong \;  \mathcal{H}_{p}^{G}(X;\mathcal{R}) & \Longrightarrow & K_{\ast}^{G}(X).
\end{eqnarray*}

Notice that, if $\dim(X)=2$ (which is the case when $X = \underbar EG$ for a Bianchi group $G$, cf. below), we have that Bredon cohomology groups (respectively homology groups)  are trivial for $p>2$, so both spectral sequences collapse in $E_2$. In fact, both spectral sequences induce  split short exact sequences (natural in $X$)
\begin{equation}\label{eq1}
0 \longrightarrow \mathcal{H}_G^{2}(X;\mathcal{R}) \longrightarrow K_G^{0}(X) \longrightarrow \mathcal{H}_G^{0}(X;\mathcal{R}) \longrightarrow 0.
\end{equation} and
\begin{equation}\label{eq2}
0 \longrightarrow \mathcal{H}^G_{2}(X;\mathcal{R}) \longrightarrow K^G_{0}(X) \longrightarrow \mathcal{H}^G_{0}(X;\mathcal{R}) \longrightarrow 0\end{equation}
and natural isomorphisms in $X$, $K_G^1(X)\cong\mathcal{H}_G^1(X;\mathcal{R})$ and $K^G_1(X)\cong\mathcal{H}^G_1(X;\mathcal{R})$. Both the exact sequences and the isomorphisms above are compatible with the induction structure.

\section{Hecke algebras and Hecke operators}

In the case where $G$ is an arithmetic group, given as a discrete subgroup  $G \subset \mathfrak{G}$ of a Lie group $\mathfrak{G}$, homogeneous $\mathfrak{G}$-spaces lead to  $G$-spaces of arithmetic relevance\footnote{Here the Lie group $\mathfrak{G}$ is the group of real or complex points of an algebraic group defined over  $\mathbb{Q}$.}. 

As many of the arithmetic properties of the group $G$ can be encoded in terms of its Hecke algebra, it will be important for us to study the action of the Hecke algebra on the Bredon cohomology and homology of these $G$-spaces. In the following paragraphs we discuss some generalities on Hecke algebras. 

\subsection{Double cosets and Hecke algebras}

Let $\mathfrak{G}$ be a group. Two subgroups of $\mathfrak{G}$ are said to be \emph{commensurable} if their intersection has finite index in both. Commensurability defines an equivalence relation in the set of subgroups of $\mathfrak{G}$. If $G_{1}$ and $G_{2}$ are subgroups of $\mathfrak{G}$ that are commensurable, we use the notation 
\begin{eqnarray*}
	G_{1} &\sim& G_{2}.
\end{eqnarray*}

Let $G$ be a subgroup of $\mathfrak{G}$. We define the \emph{commensurator of $G$ in $\mathfrak{G}$} as the subgroup
\begin{eqnarray*}
\mathrm{Comm}_{\mathfrak{G}}(G) =  &   \{ \,  g \in \mathfrak{G} \; | \;     G    \sim   g G g^{-1} \, \}.
\end{eqnarray*}
Note that if $G_{1}$ and  $G_{2}$ are commensurable subgroups of $\mathfrak{G}$ then
\begin{eqnarray*}
\mathrm{Comm}_{\mathfrak{G}}(G_1)  &=&  \mathrm{Comm}_{\mathfrak{G}}(G_2) .
\end{eqnarray*}

\begin{example} Let  
	\begin{eqnarray*}
		\mathfrak{G} &=&   \PGL_{2}^{+}(\mathbb{R})
	\end{eqnarray*}
and let $G$ be the modular group  $\PSL_2(\mathbb{Z})$, then the commensurator of $G$ in  $\mathfrak{G}$ is 
	\begin{eqnarray*}
		\mathrm{Comm}_{\mathfrak{G}}(G)  \; = \; \mathrm{Comm}_{\PGL_{2}^{+}(\mathbb{R})}(\PSL_2(\mathbb{Z})) &=&  \PGL_{2}^{+}(\mathbb{Q}).
	\end{eqnarray*}
\end{example}

\begin{example} 
Let 
\begin{eqnarray*}
		\mathfrak{G} &=&  \PGL_{2}(\mathbb{C})
\end{eqnarray*}
and let $\Q(\sqrt{-D}) \subset \mathbb{C}$ be a quadratic imaginary extension of $\mathbb{Q}$. Denote by $\cO_{\Q(\sqrt{-D})}$ the ring of integers of $\Q(\sqrt{-D})$.  If  $G = \Gamma_{D}$ is  the corresponding Bianchi group
\begin{eqnarray*}
	\Gamma_{D} &=& \PSL_2(\cO_{\Q(\sqrt{-D})})
\end{eqnarray*}
then the commensurator of $G$ in  $\mathfrak{G}$ is 
	\begin{eqnarray*}
		\mathrm{Comm}_{\mathfrak{G}}(G)  \; = \; \mathrm{Comm}_{\PGL_{2}(\mathbb{C})}(	\Gamma_{D} )  &=& \PGL_2\left(\Q(\sqrt{-D})\right).
	\end{eqnarray*}
\end{example}

As above let $\mathfrak{G}$ be a group and let $G$ be a subgroup of $\mathfrak{G}$. Given an element $g$ in 	$\mathrm{Comm}_{\mathfrak{G}}(G)$ we consider the double coset in $\mathfrak{G}$ given by
\begin{eqnarray*}
	G g G.
\end{eqnarray*}
The left action of $G$ on $G g G$ has a finite number of orbits. To compute this number. let $ G^{(g)}  =  g^{-1} G g \bigcap G $  and notice that the map
\begin{eqnarray*}
	G & \longrightarrow & G g G \\
	\gamma & \longmapsto &  g \gamma
\end{eqnarray*}
induces a surjection from $G $ to the quotient $G \backslash G g G $ with kernel $G^{(g)} $ so we have 
\begin{eqnarray*}
	G g G & = & \bigsqcup_{i=1}^{d} G \alpha_i, \quad \text{ where } d = [G : G^{(g)} ].
\end{eqnarray*}
We obtain an analogous decomposition by considering the right  action of $G$ on $G g G$.

The above decompositions lead to a natural product of double cosets as follows.  If  $\alpha$ and $\beta$ are elements in $\mathrm{Comm}_{\mathfrak{G}}(G)$ with 
\begin{eqnarray*}
	G \alpha G =  \bigsqcup_{i=1}^{d} G \alpha_i &\text{ and } &  G \beta  G  =  \bigsqcup_{j=1}^{e}  \beta_j G
\end{eqnarray*}
then we define 
\begin{eqnarray*}
	(G \alpha G) \cdot   (G \beta G) & = & \sum c_{\alpha, \beta}^{\delta}   G \delta G
\end{eqnarray*}
where 
\begin{eqnarray*}
	c_{\alpha, \beta}^{\delta}  & = & \text{ number of pairs of indices }  (i, j )    \text{ such that }  G \alpha_i  \beta_j   = G \delta  
\end{eqnarray*}
and the expression on the right is viewed as an element of the free abelian group generated by double cosets of the form $G \delta G$ with  $\delta \in \mathrm{Comm}_{\mathfrak{G}}(G)$.

If $\Delta$  is a subsemigroup of $\mathrm{Comm}_{\mathfrak{G}}(G)$  with 
\begin{eqnarray*}
	G & \subseteq  & \Delta
\end{eqnarray*}
we denote by 
\begin{eqnarray*}
	\cA (G ; \Delta)
\end{eqnarray*}
the free abelian group generated by double cosets of the form 
\begin{eqnarray*}
	G g G 
\end{eqnarray*}
with $g \in \Delta$. 

The above product of double cosets can be extended linearly to a bilinear map of $\mathbb{Z}$-modules
\begin{eqnarray*}
\cA (G ; \Delta) \times \cA (G ; \Delta)& \longrightarrow & \cA (G ; \Delta)
\end{eqnarray*}
which is associative. The group $\cA (G ; \Delta)$ thus becomes a ring which will be called \emph{the Hecke algebra of $G$ with respect to $\Delta$}. For the above results together with generalities on Hecke algebras see chapter 3 of \cite{Shimura71}.  In the case $\Delta = 	\mathrm{Comm}_{\mathfrak{G}}(G)$ we will denote $\cA (G ; \Delta)$ simply by $\cA (G)$.

\subsection{Action of Hecke operators on group cohomology}

Let $\mathfrak{G}$ be a group and let $G$ be a subgroup of $\mathfrak{G}$. As above, let $\Delta$ be a subsemigroup of $\mathrm{Comm}_{\mathfrak{G}}(G)$  with $G \subseteq \Delta$. If $M$ is an abelian group on which $\Delta$ acts by endomorphisms we can consider $M$ as a $G$-module. Elements of the Hecke algebra $\cA (G ; \Delta)$ define endomorphisms of the cohomology groups of $G$ with coefficients in $M$ leading to an action of $\cA (G ; \Delta)$ on $H^{n}(G; M)$. These operators will be called \emph{Hecke operators associated to} $(G; \Delta)$. In order to define the Hecke operator corresponding to a double coset $G g G$ with $g \in \Delta$ notice that if we have a decomposition 
\begin{eqnarray*}
	G g  G & = &  \bigsqcup_{i=1}^{d} G \alpha_i, \qquad  \alpha_i \in \Delta,
\end{eqnarray*}
and $m \in M^{G}$ is an element of $M$ fixed by $G$ then the element of $M$ given by 
\begin{eqnarray*}
	m \, | \,  G g  G & = &  \sum_{i=1}^{d}\,  \alpha_i \, m 
\end{eqnarray*}
is again fixed by $G$ and independent of the representatives $\alpha_i$ so the coset $G g  G$ defines a map
\begin{eqnarray*}
	T_{g} \, : \, M^{G} & \longrightarrow & M^{G}.
\end{eqnarray*}
These maps can be linearly extended to define operators associated to elements of $\cA(G; \Delta)$:
\begin{eqnarray*}
	m \, | \,  \xi & = &  \sum_{k=1}^{r} c_k \, T_{g_k}(m) 
\end{eqnarray*}
for $ \sum_{k=1}^{r} c_k  ( G g_k  G )  \in \cA (G; \Delta)$ 
and $m\in M^{G}$.  These operators define an action of $\cA(G; \Delta)$ on  $M^G$. 

By naturality the action of $\cA (G; \Delta)$ on $M^{G}$ extends to an action on the cohomology groups of $G$ with coefficients in $M$. At the level of group cocycles in the standard complex the action can be described as follows: 
If $g \in \Delta$ 
\begin{eqnarray*}
	G g  G & = &  \bigsqcup_{i=1}^{d} G \alpha_i,
\end{eqnarray*}
and for $\gamma \in  G$ we denote by $\sigma_{g}^{\gamma}$ the unique permutation in $S_{d}$ such that
\begin{eqnarray*}
	G \, \alpha_i  \gamma & = & G \,  \alpha_{\sigma^{g}_{\gamma}(i)} .
\end{eqnarray*}
In this way we obtain maps 
\begin{eqnarray*}
	\rho^{g}_{j} \, : \,  G  & \longrightarrow  & G   \qquad \text{ for } j = 1, \dots,d 
\end{eqnarray*}
where for $\gamma \in G$ the element $\rho^{g}_{j}  (\gamma) \in G$ is determined by 
\begin{eqnarray*}
	\alpha_j  \gamma  & =  & \rho^{g}_{j}  (\gamma) \alpha_{\sigma^{g}_{\gamma}(j)}. 
\end{eqnarray*}
Now, given a group $r$-cocycle  
\begin{eqnarray*}
	\phi \, : \, G \times \dots \times G &\longrightarrow & M
\end{eqnarray*}
we can take  
\begin{eqnarray*}
	(\, T_g \, \phi  \,) \,  (\gamma_0, \gamma_1, \dots,  \gamma_r)  & = &  \sum_{j=1}^{d} \alpha_{j}^{-1} \phi  ( \rho^{g}_{j}(\gamma_0), \rho^{g}_{j}( \gamma_1), \dots,  \rho^{g}_{j}( \gamma_r))  
\end{eqnarray*}
which is again a cocycle, so $T_g$ defines a morphism 
\begin{eqnarray*}
	T_g   \, : \,  H^{r}(G; M) & \longrightarrow & H^{r}(G; M) 
\end{eqnarray*}
which can be extended to an action of $\cA (G; \Delta)$ on $ H^{r}(G; M) $ for $r \geq 0$, which for $r = 0$, where we have $H^{0}(G; M) = M^{G} $, coincides with the action defined above.

Further information on this action together with its functorial properties and its relation to the classical theory of Hecke operators can be found in \cite{KugaEtAl81}.

\subsection{Hecke correspondences}
\label{HeckeCor}

In order to extend the above definition of Hecke operators to the Bredon cohomology groups of an arithmetic discrete group $G$ it will be important to understand the natural action of elements of $\cA (G; \Delta)$ on $G$-spaces and their quotients. The natural framework for this comes from considering elements of the Hecke ring as correspondences defined in terms of quotients corresponding to double cosets. 

As in the previous sections let  $G$ be a subgroup of a group $\mathfrak{G}$. Suppose now that the group  $\mathfrak{G}$ acts on a topological space $S$ and consider the action of the subgroup $G$ on  $S$. We will be interested in the case where $\mathfrak{G}$ is a Lie group and $S$ is a homogeneous $\mathfrak{G}$-space, also we assume in the following that the action of the discrete group $G$ on $S$ satisfies sufficient conditions for the quotient $ S / G $ to be well behaved.  

Given an element $g \in \mathrm{Comm}_{\mathfrak{G}}(G)$ consider the groups
\begin{eqnarray*}
	K \; = \;  g^{-1} G  g  \, \cap \,   G  &\text{ and  }& {}_{g}K  \; = \;   G   \, \cap \,g G  g^{-1} .
\end{eqnarray*}
Observe that we have group morphisms 
\begin{eqnarray*}
	\begin{tikzcd}
	K   \arrow[r] \arrow[d, hook]  &  {}_{g}K  \arrow[d, hook] \\
	G  &   G
	\end{tikzcd}
\end{eqnarray*}
where the horizontal arrow is a group  isomorphism and the vertical ones are inclusions with finite index.

These morphisms induce maps between the corresponding  quotients of $S$,
\begin{eqnarray*}
	\begin{tikzcd}
	S /  K   \arrow[r] \arrow[d, two heads]{a}  & S  / {}_{g}K    \arrow[d, two heads] \\
	S /  G  &   S  /  G
	\end{tikzcd}
\end{eqnarray*}
where the horizontal arrow is a homeomorphism and the vertical ones are finite index covers.  This diagram determines a correspondence 
\begin{eqnarray*}
	\mathcal{C}_{GgG} & \subset &     S  /  G \;   \times   \;  S  /  G 
\end{eqnarray*}
homeomorphic to $\mathcal{S} /  K $. We call this correspondence the \emph{Hecke correspondence from $S /  G$  to $S /  G$ associated to $GgG$}. We extend this definition using linearity in order to associate correspondences to elements of $\cA (G; \Delta)$. These can in turn be used to define a Hecke action at the level of sheaves on these spaces and their cohomology by defining operators $T_g$ via successive pullbacks and pushforwards along the horizontal maps above.  

\begin{example} The classical example leading to Hecke operators acting on modular forms corresponds to $\mathfrak{G} =   \PGL_{2}^{+}(\mathbb{R})$
viewed as a group of transformations of the hyperbolic plane $S =  \mathbb{H}_{2}$. In this case if $G$ is a subgroup of $\mathfrak{G}$ commensurable with   $\PSL_2(\mathbb{Z})$ the above quotients are modular curves and the Hecke correspondences coming from them define Hecke operators between spaces of modular forms. 
\end{example}

\begin{example} 
As mentioned at the beginning of this section the main example for us arises from 
\begin{eqnarray*}
	\mathfrak{G} &=&  \PGL_{2}(\mathbb{C})
\end{eqnarray*}
which we will view in what follows as a group of transformations of the hyperbolic $3$-space $ S =  \mathbb{H}_{3}$. Hecke correspondences for quotients of $\mathbb{H}_{3}$ by a Bianchi group
	\begin{eqnarray*}
		\Gamma_{D} &=& \PSL_2(\cO_{\Q(\sqrt{-D})})
	\end{eqnarray*}
and its subgroups will be used in the following sections in order to define Hecke operators on the Bredon cohomology of Bianchi groups. 
\end{example}

For more on this point of view together with its relation to the theory of automorphic forms see \cite{Harder87} and \cite{Harder91}.

\section{Hecke operators on Bredon (co)homology and $K$-theory}

\subsection{Restriction and corestriction}

Let us start defining the restriction and corestriction operators in equivariant K-(co)homology and Bredon (co)homology.

Let $X$ be a $G$-CW-complex and let $H\subseteq G$ be a subgroup of finite index, first note that we have natural isomorphisms given by the induction structure
\begin{eqnarray*}
K_*^H(X) & \xrightarrow{\ind_H^G} & K_*^G(G/H\times X), \\
K^*_H(X)& \xrightarrow{\ind_H^G} & K^*_G(G/H\times X), \\
\cH_*^H(X;\mathcal{R}) & \xrightarrow{\ind_H^G} & \cH_*^G(G/H\times X;\mathcal{R}), \\
\cH^*_H(X;\mathcal{R}) & \xrightarrow{\ind_H^G} & \cH^*_G(G/H\times X;\mathcal{R}).
\end{eqnarray*}

Define the \emph{corestriction} morphisms as the compositions
$$K_*^H(X)\xrightarrow{\ind_H^G}K_*^G(G/H\times X)\xrightarrow{\pi_{2*}}K_*^G(X),$$
$$\cH_*^H(X;\mathcal{R})\xrightarrow{\ind_H^G} \cH_*^G(G/H\times X;\mathcal{R})\xrightarrow{\pi_{2*}} \cH_*^G(X;\mathcal{R}),$$
$$K^*_H(X)\xrightarrow{\ind_H^G}K^*_G(G/H\times X)\xrightarrow{\pi_{2!}}K^*_G(X),$$
$$\cH^*_H(X;\mathcal{R})\xrightarrow{\ind_H^G} \cH^*_G(G/H\times X;\mathcal{R})\xrightarrow{\pi_{2!}} \cH^*_G(X;\mathcal{R}),$$
where $\pi_2:G/H\times X\to X$ is the second projection, and $\pi_{2!}$ denotes the shriek map that in this case can be defined as the composition
$$K^*_G(G/H\times X)\xrightarrow{p^*}K_G^*(\mathbb{C}[G/H]\times X)\xrightarrow{Th}K_G^*(X), $$where $\mathbb{C}[G/H]$ denotes the complex vector space generated by $G/H$ with the canonical $G$-action, $p:\mathbb{C}[G/H]\times X\to G/H\times X$ is the natural projection and $Th$ is the Thom isomorphism as in Thm. 3.14 in \cite{LueckOliver01-1}. We denote the above compositions by $\cores_H^G$ (or simply $\cores$ if $H$ and $G$ are clear from the context). This construction can be performed also in the case of Bredon cohomology.
\begin{remark}
	If $G$ is finite and $X=\pt$, corestriction corresponds to the usual induction morphism on the representation ring.
\end{remark}

We have also \emph{restriction} morphisms denoted by $\res_H^G$ (or simply by res if $H$ and $G$ are clear from the context)
\begin{eqnarray*}
K_*^G(X)& \xrightarrow{\res_H^G} & K_*^H(X), \\
\cH_*^G(X;\mathcal{R})& \xrightarrow{\res_H^G} & \cH_*^H(X;\mathcal{R}), \\
K^*_G(X) & \xrightarrow{\res_H^G} &K^*_H(X), \\
\cH^*_G(X;\mathcal{R})& \xrightarrow{\res_H^G} & \cH^*_H(X;\mathcal{R})
\end{eqnarray*}
defined in a similar way as corestriction.

\subsection{Conjugation and Hecke operators}

Now suppose that the discrete group $G$ is given as a subgroup of a Lie group $\mathfrak{G}$. Given an element $g \in \mathrm{Comm}_{\mathfrak{G}}(G)$ there are \emph{conjugation} morphisms
\begin{eqnarray*}
K_*^H(X) & \xrightarrow{Ad_g} & K_*^{gHg^{-1}}(X), \\
\cH_*^H(X;\mathcal{R})& \xrightarrow{Ad_g} & \cH_*^{gHg^{-1}}(X;\mathcal{R}), \\
K^*_H(X)& \xrightarrow{Ad_g} & K^*_{gHg^{-1}}(X), \\
\cH^*_H(X;\mathcal{R}) & \xrightarrow{Ad_g} & \cH^*_{gHg^{-1}}(X;\mathcal{R}),
\end{eqnarray*}
induced from the conjugation morphism from $\Or(H)$ to $\Or(gHg^{-1})$.

\begin{definition}
Let $G$ be a discrete subgroup of a Lie group $\mathfrak{G}$ and let $X$ be a $G$-CW-complex. Given an element $g \in \mathrm{Comm}_{\mathfrak{G}}(G)$ and a finite index subgroup $H\subseteq G$ we define the Hecke operator associated to $(G, H, g, X)$ as the composition 
 $$ \xymatrix{
	K^G_n(X) \ar[r]^{\res} &K_n^H(X)\ar[r]^{Ad_g\quad}&K^{gHg^{-1}}_n(X)\ar[r]^{\;\;\cores} &K_n^G(X).}$$
We will denote this operator by $T_{g,X}$ when $G$ and $H$ are clear from the context.

Similarly, for Bredon homology, let $X$ be a proper $G$-CW-complex, we denote by $\mathcal{T}_{g,X}$ to the Hecke operator associated to $(G,H,g,X)$ defined as the composition
	$$\xymatrix{
		\cH^G_n(X;\mathcal{R}) \ar[r]^{\res} & \cH_n^H(X;\mathcal{R})\ar[r]^{Ad_g\quad}& \cH^{gHg^{-1}}_n(X;\mathcal{R})\ar[r]^{\;\;\cores} & \cH_n^G(X;\mathcal{R}).}$$
\end{definition}

Because all of these morphisms are defined by maps of spectra, they are natural on $X$, then we have the following commutative diagram.
$$\xymatrix{
	K^G_n(\underbar{E}G) \ar[d]^{\mu_{G,n}}\ar[r]^{\res} & K_n^H(\underbar E G)\ar[d]^{\mu_{H,n}}\ar[r]^{Ad_g\;\;} & K^{gHg^{-1}}_n(\underbar EG)\ar[d]^{\mu_{gHg^{-1},n}}\ar[r]^\cores & K_n^G(\underbar EG)\ar[d]^{\mu_{G,n}}\\
	K_n(C_r^*(G))\ar[r]^{\res} & K_n(C^*_r(H))\ar[r]^{Ad_g\quad} &  K_n(C^*_r(gHg^{-1}))\ar[r]^{\quad\cores} & K_n(C_r^*(G))}$$

Thus, in the case that the group $G$ satisfies the Baum-Connes conjecture, we can compute the Hecke operators defined on the $K$-theory of the reduced $C^*$-algebra of $G$ using  the Hecke operators defined on the left hand side of the Baum-Connes conjecture, the main point being that these are more suitable for computations. Summarizing we have the following.

\begin{theorem}\label{comm}
	There is a commutative diagram
	$$\xymatrix{
		K_n^G(\underbar{E}G)\ar[d]^{\mu_{G,n}}\ar[rr]^{T_{g,\uE G}} && K_n^G(\underbar{E}G)\ar[d]^{\mu_{G,n}}\\
		K_n(C_r^*(G))\ar[rr]^{T_{g,\pt}} && K_n(C_r^*(G)).}$$
\end{theorem}

We also have a relation between $T_{g,X}$ and $\mathcal{T}_{g,X}$.

\begin{theorem}\label{bredon}
Let $G$ be a discrete group such that $\underbar{E}G$ is a $G$-CW-complex of dimension at most 2. We have commutative diagrams
	$$\xymatrix{0 \ar[r] &\mathcal{H}^G_{2}(\underbar{E}G;\mathcal{R}) \ar[r]\ar[d]^{\mathcal{T}_{g,\uE G}}& K^G_{0}(\underbar{E}G) \ar[r]\ar[d]^{{T}_{g,\uE G}}& \mathcal{H}^G_{0}(\underbar{E}G;\mathcal{R})\ar[r]\ar[d]^{\mathcal{T}_{g,\uE G}} & 0\\ 0 \ar[r] &\mathcal{H}^G_{2}(\underbar{E}G;\mathcal{R}) \ar[r]& K^G_{0}(\underbar{E}G) \ar[r]& \mathcal{H}^G_{0}(\underbar{E}G;\mathcal{R})\ar[r]&0  }$$
	
	and
	$$\xymatrix{
		K_1^G(\underbar{E}G)\ar[d]^{\cong}\ar[rr]^{T_{g,\uE G}} & & K_1^G(\underbar{E}G)\ar[d]^{\cong}\\
		\cH_1^G(\underbar{E}G;\mathcal{R})\ar[rr]^{\mathcal{T}_{g,\uE G}} & & \cH_1^G(\underbar{E}G;\mathcal{R}).}$$
\end{theorem}
\begin{proof}
	 Since $\underbar{E}G$ and $G/H\times \underbar{E}G$ are proper $G$-CW-complexes of dimension at most 2, we have natural short exact sequences as in \ref{eq2} and the natural identification $K^G_1(\uE G)\cong \cH_1^G(\uE G;\mathcal{R})$.\end{proof}

We conclude this section explaining how to compute $\mathcal{T}_{g,X}$ directly from the Bredon chain complex.

First note that if
\begin{eqnarray*}
\underline{C_n(X)}\otimes_{\mathfrak{F}}\mathcal{R} & \cong & \bigoplus_\alpha R(S_\alpha),
\end{eqnarray*}
then 
\begin{eqnarray*}
\underline{C_n(G/H\times X)}\otimes_{\mathfrak{F}}\mathcal{R} & \cong & \bigoplus_\alpha \bigoplus_{\chi\in G/H}R(\chi H\chi^{-1}\cap S_\alpha).
\end{eqnarray*}

The morphism 
\begin{eqnarray*}
\res \, : \, \cH_n^G(X;\mathcal{R}) & \longrightarrow  & \cH_n^G(G/H\times X;\mathcal{R})\;\cong\;\cH_n^H(X;\mathcal{R})
\end{eqnarray*}
is induced by the restriction morphism on the representation ring 
\begin{eqnarray*}
\bigoplus_\alpha R(S_\alpha)  & \longrightarrow  & \bigoplus_\alpha \bigoplus_{\chi\in G/H}R(\chi H\chi^{-1}\cap S_\alpha).
\end{eqnarray*}

The morphism 
\begin{eqnarray*}
Ad_g:\cH_n^H(X;\mathcal{R})& \longrightarrow  & \cH_n^{gHg^{-1}}(X;\mathcal{R})
\end{eqnarray*}
is induced by the isomorphism
\begin{eqnarray*}
\bigoplus_\alpha \bigoplus_{\chi\in G/H}R(\chi H\chi^{-1}\cap S_\alpha) & \longrightarrow  &  \bigoplus_\alpha \bigoplus_{\chi\in G/gHg^{-1}}R(\chi g Hg^{-1}\chi^{-1}\cap S_\alpha).
\end{eqnarray*}

Finally, the morphism 
\begin{eqnarray*}
\cores \, : \,  \cH_n^{gHg^{-1}}(X;\mathcal{R}) \;\cong\; \cH_n^G(G/gHg^{-1}\times X;\mathcal{R}) & \longrightarrow  &  \cH_n^G(X;\mathcal{R})
\end{eqnarray*}
 is induced by the induction morphism on the representation ring
\begin{eqnarray*}
\bigoplus_\alpha \bigoplus_{\chi\in G/gHg^{-1}}R(\chi g Hg^{-1}\chi^{-1}\cap S_\alpha) & \longrightarrow  & \bigoplus_\alpha R(S_\alpha).
\end{eqnarray*}
For Bredon cohomology the Hecke operator can be described in a similar way.

As we will see in the next section, these three morphisms can also be computed directly in the Bredon homology of $X$ as an $H$-space.

\section{Computations for Bianchi groups}\label{comp}

In this section we specialize to the case where $G$ is a Bianchi group. We use Thm. \ref{comm} and Thm. \ref{bredon} in order to explicitly compute the action of the operators $T_{g,\pt}$ for the group $\Gamma_1 = \PSL_{2}(\Z[i])$ for elements $g \in \PGL_2(\cO_{\Q(i)})$ associated to primes in $\Z[i]$. As will be seen below, these techniques apply for other Bianchi groups as well.

\begin{definition}
Let $D$ be a positive square-free integer, and let $\cO_{D } = \cO_{\Q(\sqrt{-D})}$ be the ring of integers of the imaginary quadratic extension $\Q(\sqrt{-D})$. The Bianchi group associated to $D$ is defined as
\begin{eqnarray*}
\Gamma_{D} &= & \PSL_{2}(\cO_{D }) \, = \, \SL_{2}(\cO_{\Q(\sqrt{-D})})/\{\pm I\}.
\end{eqnarray*}

\end{definition}


We can describe the rings $\cO_D$ explicitly in terms of the discriminant of the quadratic field $\Q(\sqrt{-D})$. Let $\delta=\sqrt{-D}$ and $\eta=\frac{1}{2}(1+\delta)$, we have $$\cO_D = \Z[\delta]\quad\mbox{for}\quad  D \equiv1,2\!\!\mod 4,\quad\mbox{and}$$ $$\cO_D =\Z[\eta]\quad\mbox{for}\quad D \equiv3\!\!\mod 4.$$ For a proof see \cite[Chapter 13]{Artin11}.

For $D = 1, \, 2,\,  3,\, 7, \, 11$ the ring $\cO_{D }$ is an Euclidean domain. For these values of $D$ we will refer to  the corresponding Bianchi groups $\Gamma_{D}$ as {\emph{Euclidean Bianchi groups}}. 

In general, except from $D=3$, Bianchi groups can be described as amalgamated products. The amalgam decompositions for the Euclidean Bianchi groups are described below.

\begin{proposition}
We have 
\begin{eqnarray*}
\Gamma_{1} & \cong & \big(A_{4}\ast_{C_{3}}S_{3}\big)\ast_{\PSL_{2}(\Z)}\big(S_{3}\ast_{C_{2}}D_{2}\big);
\end{eqnarray*}
\begin{eqnarray*}
\Gamma_{2} & \cong & G_{1}\ast_{(\Z*\,C_{2})}G_{2},
\end{eqnarray*}
where $G_{1}$ is the HNN extension of $C_{2}\times C_{2}$ associating two generators and $G_{2}$ is the HNN extension of $A_{4}$ associating two $3$-cycles;
\begin{eqnarray*}
\Gamma_{7} & \cong & \big(\Z\ast\,C_{2}\big)\ast_{(\Z\ast C_{2}\ast C_{2})}G_,
\end{eqnarray*}
where $G$ is the HNN extension of $S_{3}*_{C_{2}}S_{3}$ associating a $3$-cycle with itself; and
\begin{eqnarray*}
\Gamma_{11} & \cong & \big(\Z\ast\,C_{3}\big)\ast_{(\Z\ast C_{3}\ast C_{3})}G,
\end{eqnarray*}
where $G$ is the HNN extension of $A_{4}*_{C_{3}}A_{4}$ associating a $3$-cycle with itself. 
\end{proposition}

For a proof of the above facts the reader may consult \cite{Fine89}.

We will now focus on $\Gamma_1$ and the explicit computation of Hecke operators associated to it.

\subsection{The group $\Gamma_1 =  \PSL_{2}(\Z[i])$ }

From the above proposition we have an isomorphism
$$\Gamma_{1}\cong\big(A_{4}\ast_{C_{3}}S_{3}\big)\ast_{\PSL_{2}(\Z)}\big(S'_{3}\ast_{C_{2}}D_{2}\big),$$ (where for a group $G$ the notation $G'$ means just an isomorphic copy of $G$)
with $\PSL_{2}(\Z)=C'_{3}\ast C'_{2}$ and the intersections $A_{4}\cap S'_{3}=C'_{3}$, $A_{4}\cap D_{2}=\{1\}$, $S_{3}\cap S'_{3}=\{1\}$, and $S_{3}\cap D_{2}=C'_{2}$. In fact, we can obtain the presentation
$$\Gamma_{1}=\langle\; \mathbf{a,b,c,d} \;|\; \mathbf{a}^{3}=\mathbf{b}^{2}=\mathbf{c}^{3}=\mathbf{d}^{2}=(\mathbf{ac})^{2}=(\mathbf{ad})^{2}=(\mathbf{bd})^{2}=(\mathbf{bc})^{2}=1\;\rangle,$$
with $A_4=\langle \mathbf{a,c}\rangle$, $S_3=\langle \mathbf{a,d}\rangle$, $D_2=\langle \mathbf{b,d}\rangle$, and $S_3'=\langle \mathbf{b,c}\rangle$, so that $C_3=\langle \mathbf{a}\rangle$, $C_2=\langle \mathbf{b}\rangle$, $C'_3=\langle \mathbf{c}\rangle$, and $C'_2=\langle \mathbf{d}\rangle$, 
and explicit matrices that represent the generators, namely 
$$\mathbf{a}=\begin{pmatrix} 0 & i \\ i & 1\end{pmatrix},\;\; \mathbf{b}=\begin{pmatrix} 0 & i \\ i & 0\end{pmatrix},\;\; \mathbf{c}=\begin{pmatrix} 1 & 1 \\ -1 & 0\end{pmatrix},\;\;\mbox{and }\;\; \mathbf{d}=\begin{pmatrix} 0 & -1 \\ 1 & 0\end{pmatrix}.$$

\subsection{Classifying space for proper actions}

Using the above presentation and explicit generators, we can construct a 2-dimensional $\Gamma_1$-CW-complex $X$, which is a model for $\uE\Gamma_1$. Let 
$$X^{(0)}=\Gamma_{1}/A_{4}\times\{p\}\;\bigsqcup\; \Gamma_{1}/S_{3}\times\{q\}\;\bigsqcup\; \Gamma_{1}/D_{2}\times\{r\}\;\bigsqcup\; \Gamma_{1}/S'_{3}\times\{s\},$$
where each $\mathbb{D}^0$ (point) has been labeled with a letter. The $1$-skeleton is obtained from the pushout
$$\begin{tikzpicture}[scale=1, decoration={markings, mark=at position 1 with {\arrow{stealth}}}]
\node at (0,2) {$\Gamma_{1}/C_{3}\times\mathbb{S}^0\,\bigsqcup\, \Gamma_{1}/C'_{2}\times\mathbb{S}^0\,\bigsqcup\, \Gamma_{1}/C_{2}\times\mathbb{S}^0\,\bigsqcup\, \Gamma_{1}/C'_{3}\times\mathbb{S}^0$};
\node at (8,2.1) {$X^{(0)}$};
\node at (0,0) {$\Gamma_{1}/C_{3}\times\mathbb{D}^1\,\bigsqcup\, \Gamma_{1}/C'_{2}\times\mathbb{D}^1\,\bigsqcup\, \Gamma_{1}/C_{2}\times\mathbb{D}^1\,\bigsqcup\, \Gamma_{1}/C'_{3}\times\mathbb{D}^1$};
\node at (8,0.1) {$X^{(1)}$};
\draw[postaction={decorate}] (5.2,2) -- (7.3,2);
\node at (6,2.3) {$\varphi$};
\draw[postaction={decorate}] (0,1.5) -- (0,0.6);
\node at (-0.8,1.1) {\footnotesize inclusion};
\draw[postaction={decorate}] (8,1.5) -- (8,0.6);
\draw[postaction={decorate}] (5.2,0) -- (7.3,0);
\end{tikzpicture}$$
so that $X^{(1)}$ is the union of $X^{(0)}$ and many copies of $\mathbb{D}^1$, identifying the image by $\varphi$ and the inclusion, respectively, of many copies of $\mathbb{S}^0$. Writing each copy of $\mathbb{S}^0$ as two ordered points $\{-1,1\}$ and denoting a point in $X^{(0)}$ just as the coset, the map $\varphi$ is defined as follows. For any $\gamma\in\Gamma_1$, 
\begin{eqnarray*}
	\varphi:& \gamma C_{3}\times\{-1,1\}\mapsto \{\gamma A_{4}\,,\,\gamma S_{3}\},\\
	&\gamma C'_{2}\times\{-1,1\}\mapsto \{\gamma S_{3}\,,\,\gamma D_{2}\},\\
	&\gamma C_{2}\times\{-1,1\}\mapsto \{\gamma D_{2}\,,\,\gamma S'_{3}\},\\
	&\gamma C'_{3}\times\{-1,1\}\mapsto \{\gamma S'_{3}\,,\,\gamma A_{4}\}.
\end{eqnarray*}

This means that we will add a segment between two points whenever their corresponding cosets intersect as a coset of any of the cyclic groups in $\Gamma_1$. Take $P$, $Q$, $R$, $S$ as the trivial cosets of $A_4$, $S_3$, $D_2$, $S'_3$, respectively. The space $X^{(1)}$ would begin to look like this:

$$
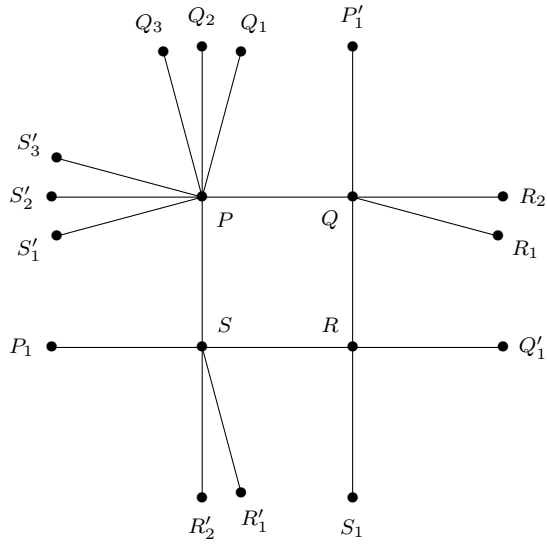
\begin{tikzpicture}[scale=1, decoration={markings, mark=at position 1 with {\arrow{stealth}}}]
\node at (0,2) {$\bullet$};
\node at (0.3,1.7) {\footnotesize $P$};
\node at (2,2) {$\bullet$};
\node at (1.7,1.7) {\footnotesize $Q$};
\node at (0,0) {$\bullet$};
\node at (0.3,0.3) {\footnotesize $S$};
\node at (2,0) {$\bullet$};
\node at (1.7,0.3) {\footnotesize $R$};
\draw (0,2) -- (2,2);
\draw (0,0) -- (0,2);
\draw (2,2) -- (2,0);
\draw (0,0) -- (2,0);

\draw (0,2) -- (0.517,3.932);
\node at (0.517,3.932) {$\bullet$};
\node at (0.7,4.3) {\footnotesize $Q_1$};
\draw (0,2) -- (0,4);
\node at (0,4) {$\bullet$};
\node at (0,4.4) {\footnotesize $Q_2$};
\draw (0,2) -- (-0.517,3.932);
\node at (-0.517,3.932) {$\bullet$};
\node at (-0.7,4.3) {\footnotesize $Q_3$};

\draw (2,2) -- (2,4);
\node at (2,4) {$\bullet$};
\node at (2,4.4) {\footnotesize $P'_1$};

\draw (0,2) -- (-1.932,1.482);
\node at (-1.932,1.482) {$\bullet$};
\node at (-2.3,1.3) {\footnotesize $S'_1$};
\draw (0,2) -- (-2,2);
\node at (-2,2) {$\bullet$};
\node at (-2.4,2) {\footnotesize $S'_2$};
\draw (0,2) -- (-1.932,2.518);
\node at (-1.932,2.518) {$\bullet$};
\node at (-2.3,2.7) {\footnotesize $S'_3$};

\draw (0,0) -- (-2,0);
\node at (-2,0) {$\bullet$};
\node at (-2.4,0) {\footnotesize $P_1$};

\draw (2,2) -- (3.932,1.482);
\node at (3.932,1.482) {$\bullet$};
\node at (4.3,1.3) {\footnotesize $R_1$};
\draw (2,2) -- (4,2);
\node at (4,2) {$\bullet$};
\node at (4.4,2) {\footnotesize $R_2$};

\draw (2,0) -- (4,0);
\node at (4,0) {$\bullet$};
\node at (4.4,0) {\footnotesize $Q'_1$};

\draw (0,0) -- (0.517,-1.932);
\node at (0.517,-1.932) {$\bullet$};
\node at (0.7,-2.3) {\footnotesize $R'_1$};
\draw (0,0) -- (0,-2);
\node at (0,-2) {$\bullet$};
\node at (0,-2.4) {\footnotesize $R'_2$};

\draw (2,0) -- (2,-2);
\node at (2,-2) {$\bullet$};
\node at (2,-2.4) {\footnotesize $S_1$};
\end{tikzpicture}$$

The lines $PQ_i$, $i=1,2,3$, come from the cosets $\mathbf{c}C_3$, $\mathbf{c}^2 C_3$, and $\mathbf{ac}^2 C_3$, re\-spec\-tive\-ly. There are no more cosets of $S_3$ connected to $A_4$. It continues similarly.

Finally, we add a 2-cell, filling the square: 
$$\begin{tikzpicture}[scale=1, decoration={markings, mark=at position 1 with {\arrow{stealth}}}]
\node at (0,2) {$\Gamma_1/\{1\}\times \mathbb{S}^{1}$};
\node at (5,2) {$X^{(1)}$};
\node at (0,0) {$\Gamma_1/\{1\}\times \mathbb{D}^{2}$};
\node at (5.5,0) {$X^{(2)}=X$};
\draw[postaction={decorate}] (1.3,2) -- (4.4,2);
\draw[postaction={decorate}] (0,1.6) -- (0,0.5);
\draw[postaction={decorate}] (5,1.6) -- (5,0.5);
\draw[postaction={decorate}] (1.3,0) -- (4.5,0);
\end{tikzpicture}.$$

This space is proper since all the isotropy groups are finite groups, and this is because $X$ can be thought as the space obtained from a square by the action of $\Gamma_1$ with the isotropy groups showed below.
$$
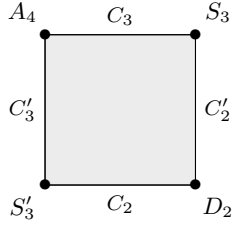
\begin{tikzpicture}[scale=1, decoration={markings, mark=at position 1 with {\arrow{stealth}}}]
\draw[fill=gray!15] (0,0) rectangle (2,2);

\node at (0,2) {$\bullet$};
\node at (-0.3,2.3) {\small $A_4$};
\node at (2,2) {$\bullet$};
\node at (2.3,2.3) {\small $S_3$};
\node at (0,0) {$\bullet$};
\node at (-0.3,-0.3) {\small $S'_3$};
\node at (2,0) {$\bullet$};
\node at (2.3,-0.3) {\small $D_2$};
\draw (0,2) -- (2,2);
\node at (1,2.25) {\footnotesize $C_3$};
\draw (0,0) -- (0,2);
\node at (-0.3,1) {\footnotesize $C'_3$};
\draw (2,2) -- (2,0);
\node at (2.3,1) {\footnotesize $C'_2$};
\draw (0,0) -- (2,0);
\node at (1,-0.25) {\footnotesize $C_2$};
\end{tikzpicture}$$

Also, $X$ is indeed a model for $\underline{E}\Gamma_1$, since every fixed space $X^H$, $H$ finite subgroup of $\Gamma_1$, is weakly contractible.

An alternative description of this space can be found in \cite{Floge83} and \cite{Rahm10}.

\subsection{Bredon cohomology for $\Gamma_{1}$}

In order to compute the Hecke operators $T_{g,\pt}$ we start computing the Bredon (co)homology groups of $\Gamma_{1}$ with coefficients in the representation ring. We will later use these groups together with  Thm. \ref{bredon} to obtain computations in equivariant K-homology.

The Bredon cochain complex with coefficients in the representation ring \ref{bredonc} for the group $\Gamma_{1}$ and the space $X=\underline{E}\Gamma_1$ has the form
$$0\longrightarrow \bigoplus_{\substack{ \alpha \\ 0\text{-cells} }} R(S_{\alpha}) \xrightarrow{\;d^0\;} \bigoplus_{\substack{ \alpha \\ 1\text{-cells} }} R(S_{\alpha}) \xrightarrow{\;d^1\;} \bigoplus_{\substack{ \alpha \\ 2\text{-cells} }} R(S_{\alpha}) \longrightarrow0,$$
where the sum runs over representatives of $n$-cells, the $S_\alpha$ are the corresponding stabilizers, and the differentials are given by restriction of representations. We know that 
$$R(A_{4})\cong\Z^4,\quad R(S_3)\cong R(S'_{3})\cong \Z^3,\quad R(D_2)\cong\Z^4,\quad\mbox{and}\quad R(C_n)\cong\Z^n,$$
so the cochain complex becomes 
$$0\longrightarrow \Z^{4+3+4+3} \xrightarrow{\quad d^0\quad} \Z^{3+2+2+3} \xrightarrow{\quad d^1\quad} \Z \longrightarrow0.$$
Here, $d^1$ is represented by the matrix $(\;1\;1\;1\;1\;1\;1\;1\;1\;1\;1\;),$ of rank $1$, and $d^0$ by the matrix
$$\left(\begin{array}{cccc:ccc:cccc:ccc}
-1 &  0 &  0 & -1 &  1 &  1 &  0 &    &    &    &    &    &    &    \\
0 & -1 &  0 & -1 &  0 &  0 &  1 &    &    &    &    &    &    &    \\
0 &  0 & -1 & -1 &  0 &  0 &  1 &    &    &    &    &    &    &    \\ \hdashline
&    &    &    & -1 &  0 & -1 &  1 &  1 &  0 &  0 &    &    &    \\
&    &    &    &  0 & -1 & -1 &  0 &  0 &  1 &  1 &    &    &    \\ \hdashline
&    &    &    &    &    &    & -1 &  0 & -1 &  0 &  1 &  0 &  1 \\
&    &    &    &    &    &    &  0 & -1 &  0 & -1 &  0 &  1 &  1 \\ \hdashline
1 &  0 &  0 &  1 &    &    &    &    &    &    &    & -1 & -1 &  0 \\
0 &  0 &  1 &  1 &    &    &    &    &    &    &    &  0 &  0 & -1 \\
0 &  1 &  0 &  1 &    &    &    &    &    &    &    &  0 &  0 & -1
\end{array}\right),$$
of rank $8$ (blank spaces stand for blocks of zeros).

We obtain 
$$\mathcal{H}_{\Gamma_1}^{n}(X;\mathcal{R}) \cong \begin{cases}
\Z^6, & n=0; \\
\Z, & n=1; \\
0, & n\geq2.
\end{cases}$$\\

Bredon homology is computed in a similar way. We have 
$$\mathcal{H}_n^{\Gamma_1}(X;\mathcal{R}) \cong \begin{cases}
\Z^6, & n=0; \\
\Z, & n=1; \\
0, & n\geq2.
\end{cases}$$

\subsection{Hecke correspondences  for prime level congruence subgroups}

We study now Hecke operators coming from congruence subgroups associated to primes in $\Z[i]$. We begin by reviewing these first. 

The group of units of the ring  of Gaussian integers $\Z[i]$ is given by $\{1, -1, i, -i \}$. Any prime ideal in 
$\Z[i]$ is of the form $(\pi)$ where $\pi$ is an irreducible element of $\Z[i]$. Up to units, these prime elements are the following: 
\begin{itemize}
	\item $\pi = 1 + i$, 
	\item $\pi = p$ for a prime $p$ in $\Z$ with $p \equiv 3 \mod 4$, 
	\item $ \pi =  a+ib$, with $a^2+b^2=p$ for a prime $p$ in $\Z$ with $p \equiv 1 \mod 4$.
\end{itemize}

Notice that $(2) = (1 + i)^2$ and that this is the only prime which ramifies in $\Z[i]$. This leads to exceptional behavior of the prime $1+i$. One instance  of this exceptional behavior will be seen in the computations for the classifying spaces for proper actions and the isotropy groups arising in the corresponding Hecke operators.\\

Fix now a prime $p$ in $\Z[i]$. Let $g$ be the class 
of
\begin{eqnarray*}
\begin{pmatrix} p & 0 \\ 0 & 1\end{pmatrix}
\end{eqnarray*} 
in $\PGL_2\left(\Q(i)\right)	=   \mathrm{Comm}_{\PGL_{2}(\mathbb{C})}(	\Gamma_{1} )$. 

As above, cf. \ref{HeckeCor}, in order to define the Hecke operator $T_g$ associated to the double coset $\Gamma_1 g \Gamma_1$, we have to consider the Hecke correspondence determined by the congruence subgroup 
\begin{eqnarray*}
K & = & \Gamma_1 \cap g^{-1}\Gamma_1 g.
\end{eqnarray*}

We start describing explicitly this subgroup. For this, notice that if $\gamma$ is the class of a matrix $\begin{pmatrix} a & b \\ c & d\end{pmatrix}$ in $\PGL_{2}(\mathbb{C})$ 
 we have
$$g^{-1}\gamma g=\begin{pmatrix} 1/p & 0 \\ 0 & 1\end{pmatrix}\begin{pmatrix} a & b \\ c & d\end{pmatrix}\begin{pmatrix} p & 0 \\ 0 & 1\end{pmatrix} = \begin{pmatrix} a & b/p \\ pc & d\end{pmatrix}.$$
This means that (the class of) any matrix in $K$ will be of this form. Then 
$$K=\left\{\begin{pmatrix} a & b \\ c & d\end{pmatrix}\in\Gamma_1 \quad:\quad c\in p\cdot \Z[i] \right\},$$
so $K$ is the congruence subgroup
$$K=\left\{\gamma\in\Gamma_1 \quad:\quad \gamma\equiv\begin{pmatrix} \cdot & \cdot \\ 0 & \cdot\end{pmatrix} \mbox{ mod }p \right\}.$$\\

To compute the index of $K$ in $\Gamma_1$ we use the following lemma. Here the notation $\widetilde{\text{P}}\SL_2(\mathbb{F})$ stands for $\SL_2(\mathbb{F})/\{\pm I\}$ (usually without the tilde it would be the quotient by all the center of the group).

\begin{lemma}
For a field with $q$ elements, $\mathbb{F}_q$, the order of the group $\widetilde{\emph{P}}\SL_2(\mathbb{F}_q)$ is $q(q^2-1)$ if $q$ is even and $q(q^2-1)/2$ if $q$ is odd.
\end{lemma}
\begin{proof}
	The group $\SL_2(\mathbb{F}_q)$ is the kernel of the surjective homomorphism 
	$$\mbox{det}:\GL_2(\mathbb{F}_q)\longrightarrow \mathbb{F}_q^\ast,\quad \mbox{so} \quad |\SL_2(\mathbb{F}_q)|=|\GL_2(\mathbb{F}_q)|/|\mathbb{F}_q^\ast|=|\GL_2(\mathbb{F}_q)|/(q-1).$$
	
	The order of $\GL_2(\mathbb{F}_q)$ is equal to the number of bases for $\mathbb{F}_q^2$ over $\mathbb{F}_q$ (there is a non-singular matrix for every pair of linearly independent vectors in $\mathbb{F}_q^2$), which is equal to the number of non-zero vectors in $\mathbb{F}_q^2$ times the number of vectors which are not a multiple of the first one, that is $(q^2 -1)(q^2 -q)$.
	
	Then, $|\SL_2(\mathbb{F}_q)|=q(q^2 -1)$. Finally, since $\widetilde{\text{P}}\SL_2(\mathbb{F}_q)=\SL_2(\mathbb{F}_q)/\{\pm I\}$, we divide by $2$ when $q$ is odd and we do not when $q$ is even, because the characteristic of $\mathbb{F}_q$ is $2$, so $I=-I$.
\end{proof}

It is known that the quotient $\Z[i]/p$ is a field and is isomorphic to $\mathbb{F}_{|p|}$, where $|p|$ is the norm of $p$ in $\Z[i]$ (the square of its absolute value as a complex number).\\

Consider the surjective homomorphism 
$$\pi:\Gamma_1=\PSL_2(\Z[i])\longrightarrow \widetilde{\text{P}}\SL_2(\Z[i]/p).$$
The kernel of $\pi$ is the group of matrices that are the identity modulo $p$; the index of this subgroup in $\Gamma_1$ is equal to the size of $\widetilde{\text{P}}\SL_2(\mathbb{F}_{|p|})$. Since Ker$(\pi)$ is contained in the group $K$, we have 
$$(\Gamma_1:K)=\dfrac{(\Gamma_1:\mbox{Ker}(\pi))}{(K:\mbox{Ker}(\pi))} = \dfrac{|\widetilde{\text{P}}\SL_2(\mathbb{F}_{|p|})|}{(K:\mbox{Ker}(\pi))}.$$
Besides, the index $(K:\mbox{Ker}(\pi))$ is equal to the size of the quotient group $$K\,/\,\mbox{Ker}(\pi)\cong \left\{\begin{pmatrix} a & b \\ 0 & a^{-1}\end{pmatrix}\in\widetilde{\text{P}}\SL_2(\mathbb{F}_{|p|})\right\},$$
and the size of this group is $|p|(|p|-1)$, if $|p|$ is even, or $|p|(|p|-1)/2$, if $|p|$ is odd. (As before, we do not divide by $2$ when $|p|$ is even because $\mathbb{F}_{|p|}$ has characteristic $2$.)

Thus, we obtain that $$(\Gamma_1:K)=\dfrac{|p|(|p|^2 -1)}{|p|(|p|-1)}=|p|+1.$$

Furthermore, we can give (left and right) coset representatives for $\Gamma_1$ modulo $K$. There are $|p|$ cosets represented by the matrices $$\gamma_{z}=\begin{pmatrix} 1 & 0 \\ z & 1\end{pmatrix}, \quad \mbox{with $z$ as representatives of} \quad \Z[i]/p \cong \mathbb{F}_{|p|},$$
and the last is given by the matrix $\sigma=\begin{pmatrix} 0 & -1 \\ 1 & 0\end{pmatrix}$. \\

\subsubsection{Bredon (co)homology of the congruence subgroup $K$}

Now, we will compute the Bredon (co)homology of $\underbar{E}K$. First, note that since $K$ is a subgroup of $\Gamma_1$, we can think of $X=\underline{E}\Gamma_1$ as a model for $\underline{E}K$. Then, we need $K$-orbit representatives for $n$-cells in $X$. We can start from the right coset partition 
$$\Gamma_1 = \bigsqcup_{K\gamma\,\in\, K\backslash\Gamma_1} K\gamma.$$ 
Note that for any cell $e\subset X$, the $\Gamma_1$-orbit of $e$ splits into the union of some $K$-orbits, 
$$\Gamma_1\cdot e = \bigcup_{K\gamma\,\in\, K\backslash\Gamma_1} K\gamma\cdot e,$$
and, after omitting repetitions, the union would be disjoint (apart from the boundaries). To count these repetitions, it is sufficient to find if there exists any $k\in K$ such that 
$$\gamma^{-1}k\,\gamma' \in\mbox{Stab}_{\Gamma_1}(e) \quad \mbox{for two distinct representatives $\gamma$, $\gamma'$ of $K\backslash\Gamma_1$},$$
in which case we would know that the $K$-orbits of $\gamma e$ and $\gamma'e$ are the same.\\

In the following paragraphs we focus on the case $p=1+i$, carrying out in full the corresponding computations. For other primes, an algorithm in GAP has been developed by the first author (cf. \cite{Munoz20}). 
We have that $(\Gamma_1:K)=|1+i|+1=3$ and 
$$\Gamma_1 = K\gamma_0 \;\sqcup\; K\gamma_1 \;\sqcup\; K\sigma = K \begin{pmatrix} 1 & 0 \\ 0 & 1\end{pmatrix} \;\sqcup\; K \begin{pmatrix} 1 & 0 \\ 1 & 1\end{pmatrix} \;\sqcup\; K \begin{pmatrix} 0 & -1 \\ 1 & 0\end{pmatrix}.$$
First we search for repeated $K$-orbits. These are all we need:
$$\begin{aligned}
\gamma_0^{-1} \begin{pmatrix} -i & i \\ i-1 & 1\end{pmatrix} \gamma_1 = \begin{pmatrix} 0 & i \\ i & 1\end{pmatrix} &= \; \mathbf{a}, & 
\qquad \sigma^{-1} \begin{pmatrix} i & 1 \\ 0 & -i\end{pmatrix} \gamma_0 = \begin{pmatrix} 0 & i \\ i & 1\end{pmatrix} &= \; \mathbf{a}, \\ 
\gamma_0^{-1} \begin{pmatrix} 1 & -1 \\ 0 & 1\end{pmatrix} \gamma_1 = \begin{pmatrix} 0 & -1 \\ 1 & 1\end{pmatrix} &= \; \mathbf{c}^2, & 
\sigma^{-1} \begin{pmatrix} 1 & 0 \\ 0 & 1\end{pmatrix} \gamma_1 = \begin{pmatrix} 1 & 1 \\ -1 & 0\end{pmatrix} &= \; \mathbf{c}, \\ 
\sigma^{-1} \begin{pmatrix} i & 0 \\ 0 & -i\end{pmatrix} \gamma_0 = \begin{pmatrix} 0 & i \\ i & 0\end{pmatrix} &= \; \mathbf{b}, & 
\sigma^{-1} \begin{pmatrix} 1 & 0 \\ 0 & 1\end{pmatrix} \gamma_0 = \begin{pmatrix} 0 & -1 \\ 1 & 0\end{pmatrix} &= \; \mathbf{d}, \\ 
\gamma_0^{-1} \begin{pmatrix} i & 1 \\ 0 & -i\end{pmatrix} \gamma_1 = \begin{pmatrix} 1+i & 1 \\ -i & -i\end{pmatrix} &= \; \mathbf{a}^2\mathbf{c}, & 
\sigma^{-1} \begin{pmatrix} i & i \\ 1+i & 1\end{pmatrix} \gamma_0 = \begin{pmatrix} 1+i & 1 \\ -i & -i\end{pmatrix} &= \; \mathbf{a}^2\mathbf{c}, \\ 
\gamma_0^{-1} \begin{pmatrix} i & 0 \\ 1+i & -i\end{pmatrix} \gamma_1 = \begin{pmatrix} i & 0 \\ 1 & -i\end{pmatrix} &= \; \mathbf{ad}, & 
\sigma^{-1} \begin{pmatrix} -i & i \\ i-1 & 1\end{pmatrix} \gamma_1 = \begin{pmatrix} i & 1 \\ 0 & -i\end{pmatrix} &= \; \mathbf{a}^2\mathbf{d}, \\ 
\gamma_0^{-1} \begin{pmatrix} -i & 0 \\ 0 & i\end{pmatrix} \gamma_1 = \begin{pmatrix} -i & 0 \\ i & i\end{pmatrix} &= \; \mathbf{bc}, & 
\sigma^{-1} \begin{pmatrix} -i & i \\ 0 & i\end{pmatrix} \gamma_1 = \begin{pmatrix} i & i \\ 0 & -i\end{pmatrix} &= \; \mathbf{b}\mathbf{c}^2.
\end{aligned}$$\\

For $0$-cells we have 
$$\Gamma_1 \cdot P = K \cdot P, \quad \Gamma_1 \cdot Q = K \cdot Q, \quad \Gamma_1 \cdot S = K \cdot S, \quad \mbox{and}$$
$$\Gamma_1 \cdot R = K \cdot R \;\;\sqcup\;\; K\gamma_1 \cdot R = K \cdot R \;\;\sqcup\;\; K \cdot \widetilde{R},$$ 
with $\widetilde{R}=\gamma_1 R$. For $1$-cells we have 
$$\Gamma_1 \cdot PQ = K \cdot PQ, \quad \Gamma_1 \cdot QR = K \cdot QR \;\;\sqcup\;\; K \cdot Q\widetilde{R}, $$
$$\Gamma_1 \cdot RS = K \cdot RS \;\;\sqcup\;\; K \cdot \widetilde{R}S, \quad \mbox{and} \quad  \Gamma_1 \cdot SP = K \cdot SP,$$ 
with $Q\widetilde{R}=\gamma_1 QR$ and $\widetilde{R}S = \gamma_1 RS$. Finally, the orbit of the $2$-cell is not repeated, so there are three $2$-cells in the quotient $X/K$. Let $E$ be the 2-cell $PQRS$.

The quotient space $X/K$ would look like this:
$$
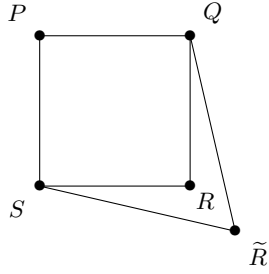
\begin{tikzpicture}[scale=1, decoration={markings, mark=at position 1 with {\arrow{stealth}}}]

\node at (0,2) {$\bullet$};
\node at (-0.3,2.3) {\small $P$};
\node at (2,2) {$\bullet$};
\node at (2.3,2.3) {\small $Q$};
\node at (0,0) {$\bullet$};
\node at (-0.3,-0.3) {\small $S$};
\node at (2,0) {$\bullet$};
\node at (2.2,-0.2) {\small $R$};
\draw (0,2) -- (2,2);
\draw (0,0) -- (0,2);
\draw (2,2) -- (2,0);
\draw (0,0) -- (2,0);

\node at (2.6,-0.6) {$\bullet$};
\node at (2.9,-0.9) {\small $\widetilde{R}$};
\draw (2,2) -- (2.6,-0.6);
\draw (0,0) -- (2.6,-0.6);

\node at (7,1.3) {\small With two 2-cells $E$ and $\sigma E$ with the same};
\node at (7,0.8) {\small boundary, and one other 2-cell $\gamma_1 E$.};

\end{tikzpicture}$$

The stabilizer of each orbit representative is the intersection between the stabilizer in $\Gamma_1$ and the subgroup $K$, so
$$\mbox{Stab}_K(P)=A_4\cap K=\langle\mathbf{ac,ca,ac}^2\mathbf{a}\rangle\cong D_2,$$
$$\mbox{Stab}_K(Q)=S_3\cap K=\langle\mathbf{a}^2\mathbf{d}\rangle\cong C_2, \quad \mbox{Stab}_K(R)=D_2\cap K=\langle\mathbf{bd}\rangle\cong C_2,$$
$$\mbox{Stab}_K(S)=S'_3\cap K=\langle\mathbf{bc}^2\rangle\cong C_2,$$ 
and 
$$\mbox{Stab}_K(\widetilde{R})=\mbox{Stab}_{\Gamma_1}(\gamma_1 R)\cap K=\left(\gamma_1\,\mbox{Stab}_{\Gamma_1}(R)\,\gamma_1^{-1}\right)\cap K$$
$$=\left(\gamma_1\, \langle\bfb,\bfd,\bfb\bfd\rangle \,\gamma_1^{-1}\right)\cap K=\gamma_1\, \langle\bfb,\bfd,\bfb\bfd\rangle \,\gamma_1^{-1}\cong D_2.$$
The remaining stabilizers are trivial, except the following two.
$$\mbox{Stab}_K(\widetilde{R}S)=\left(\gamma_1\,\mbox{Stab}_{\Gamma_1}(RS)\,\gamma_1^{-1}\right)\cap K=\langle\gamma_1\,\bfb\,\gamma_1^{-1}\rangle\cap K\cong C_2$$
$$\mbox{Stab}_K(Q\widetilde{R})=\left(\gamma_1\,\mbox{Stab}_{\Gamma_1}(QR)\,\gamma_1^{-1}\right)\cap K=\langle\gamma_1\,\bfd\,\gamma_1^{-1}\rangle\cap K\cong C_2$$

The cochain complex becomes 
$$0\longrightarrow \Z^{4+2+2+4+2}=\Z^{14} \xrightarrow{\quad d^0\quad} \Z^{1+1+2+1+2+1}=\Z^{8} \xrightarrow{\quad d^1\quad} \Z^{3} \longrightarrow0.$$
Here, $d^0$ is represented by the matrix
$$\left(\begin{array}{l:cccc:cc:cc:cccc:cc}
& P &   &   &   & Q &   & R &   & \widetilde{R} &   &   &   & S &   \\ \hdashline
PQ & -1 & -1 & -1 & -1 
&  1 &  1 
&    &    
&    &    &   &   
&    &    \\ \hdashline
QR & &    &    &    
& -1 & -1 
&  1 &  1 
&    &    &   &   
&    &    \\ \hdashline
& &    &    &    
& -1 & 0 
&    &   
& 1 & 1 & 0 & 0 
&    &    \\
Q\widetilde{R} & &    &    &    
& 0 & -1 
&  & 
& 0 & 0 & 1 & 1 
&  &  \\ \hdashline
RS & &    &    &    
&    &    
& -1 & -1 
&   &    &   &  
&  1 &  1 \\ \hdashline
&   &   &   &   
&    &    
&   &   
& -1 & 0 & -1 & 0 
& 1 & 0 \\
\widetilde{R}S &   &   &   &   
&    &    
&    &    
& 0 & -1 & 0 & -1 
& 0 & 1 \\ \hdashline
SP & 1 &  1 &  1 &  1 
&    &    
&    &   
&   &    &    &    
& -1 & -1
\end{array}\right),$$
of rank $6$, and $d^1$ by the matrix 
$$\left(\begin{array}{c:cccccccc}
& PQ & QR & Q\widetilde{R} &  & RS & \widetilde{R}S &  & SP \\ \hdashline
E & 1 &  1 
&  0 & 0 
&  1 
&  0 & 0 
& 1 \\
\sigma E & 1 &  1 
&  0 & 0 
&  1 
&  0 & 0 
& 1 \\
\gamma_1 E & 1 &  0 
& 1 & 1 
& 0 
& 1 & 1 
& 1
\end{array}\right),$$
of rank $2$.

We obtain 
$$\mathcal{H}_{K}^{n}(X;\mathcal{R}) \cong \begin{cases}
\Z^8, & n=0; \\
0, & n=1, \; n>2; \\
\Z, & n=2.
\end{cases}$$\\

As before, the homology is computed in a similar way. We have 
$$\mathcal{H}_n^K(X;\mathcal{R}) \cong \begin{cases}
\Z^8, & n=0; \\
0, & n=1, \; n>2; \\
\Z, & n=2.
\end{cases}$$

\subsection{The Hecke operator $T_g$ in Bredon homology via the subgroup $K$}

\subsubsection{Restriction}
As we explained above, the morphism at the level of chain complexes is given by restriction of representations of the isotropy groups of each cell. We obtain the following commutative diagram.




$$\begin{tikzpicture}[scale=1, decoration={markings, mark=at position 1 with {\arrow{stealth}}}]
\node at (-2,2) {$0$};
\draw[postaction={decorate}] (-1.7,2) -- (-0.5,2);
\node at (0,2) {$\Z$};
\draw[postaction={decorate}] (0.5,2) -- (2.5,2);
\node at (3,2) {$\Z^{10}$};
\draw[postaction={decorate}] (3.5,2) -- (5.5,2);
\node at (6,2) {$\Z^{14}$};
\draw[postaction={decorate}] (6.5,2) -- (7.7,2);
\node at (8,2) {$0$};

\node at (-2,0) {$0$};
\draw[postaction={decorate}] (-1.7,0) -- (-0.5,0);
\node at (0,0) {$\Z^{3}$};
\draw[postaction={decorate}] (0.5,0) -- (2.5,0);
\node at (3,0) {$\Z^8$};
\draw[postaction={decorate}] (3.5,0) -- (5.5,0);
\node at (6,0) {$\Z^{14}$};
\draw[postaction={decorate}] (6.5,0) -- (7.7,0);
\node at (8,0) {$0$};

\node at (1.5,2.3) {$d^{\Gamma_1}_1$};
\node at (4.5,2.3) {$d^{\Gamma_1}_0$};
\node at (1.5,0.3) {$d^{K}_1$};
\node at (4.5,0.3) {$d^{K}_0$};

\draw[postaction={decorate}] (-0.1,1.7) -- (-0.1,0.4);
\node at (-0.4,1) {$f_2$};
\draw[postaction={decorate}] (2.9,1.7) -- (2.9,0.4);
\node at (2.6,1) {$f_1$};
\draw[postaction={decorate}] (5.9,1.7) -- (5.9,0.4);
\node at (5.6,1) {$f_0$};
\end{tikzpicture}$$
where $f_0$ is represented by the matrix 
$$\left(\begin{array}{c:rrrr:rrr:rrrr:rrr}
& P & & & & Q & & & R & & & & S &  & \\ \hdashline
& 1 &  1 &  1 &  0 & & & & & & & & & \\
P & 0 &  0 &  0 &  1 & & & & & & & & & \\
& 0 &  0 &  0 &  1 & & & & & & & & & \\
& 0 &  0 &  0 &  1 & & & & & & & & & \\ \hdashline
Q & & & & &  1 &  0 &  1 & & & & & & & \\
& & & & &  0 &  1 &  1 & & & & & & & \\ \hdashline
R & & & & & & & &   1 &  0 &  0 &  1 & & & \\
& & & & & & & &  0 &  1 &  1 &  0 & & & \\ \hdashline
& & & & & & & &  1 &  0 &  0 &  0 & & & \\
\widetilde{R}& & & & & & & &  0 &  1 &  0 &  0 & & & \\
& & & & & & & &  0 &  0 &  1 &  0 & & & \\
& & & & & & & &  0 &  0 &  0 &  1 & & & \\ \hdashline
S & & & & & & & & & & & &  1 &  0 &  1 \\
& & & & & & & & & & & &  0 &  1 &  1
\end{array}\right),$$
of rank $10$, and $f_1$ by 
$$\left(\begin{array}{c:ccc:cc:cc:ccc}
& PQ & & & QR & & RS & & SP & & \\
PQ & 1 & 1 & 1 & 0 & 0 & 0 & 0 & 0 & 0 & 0 \\
QR & 0 &  0 &  0 &  1 &  1 &  0 &  0 &  0 &  0 &  0 \\
Q\widetilde{R} & 0 &  0 &  0 &  1 &  0 &  0 &  0 &  0 &  0 &  0 \\
& 0 &  0 &  0 &  0 &  1 &  0 &  0 &  0 &  0 &  0 \\
RS & 0 &  0 &  0 &  0 &  0 &  1 &  1 &  0 &  0 &  0 \\
\widetilde{R}S & 0 &  0 &  0 &  0 &  0 &  1 &  0 &  0 &  0 &  0 \\
& 0 &  0 &  0 &  0 &  0 &  0 & 1 & 0 & 0 & 0 \\
SP & 0 &  0 &  0 &  0 &  0 &  0 &  0 &  1 &  1 &  1
\end{array}\right),$$
of rank $6$.


Using these matrices, the restriction morphism 
$$\text{res}:\mathcal{H}^{\Gamma_1}_{0}(X;\mathcal{R})\cong \Z^6 \longrightarrow \mathcal{H}^{K}_{0}(X;\mathcal{R})\cong \Z^8$$
can be represented by the matrix 
\begin{equation}\label{m1}\left(\begin{array}{cccccc}
1 & 2 & 0 & 0 & 0 & 0 \\
0 & 0 & 1 & 0 & 0 & 0 \\
0 & 0 & 1 & 0 & 0 & 0 \\
0 & 0 & 1 & 0 & 0 & 0 \\
0 & 1 & 1 & 1 & 0 & 0 \\
1 & 1 & 2 &-1 & 1 &-1 \\
0 & 0 & 0 & 0 & 1 & 0 \\
0 & 0 & 0 & 0 & 0 & 1
\end{array}\right).\end{equation}


\subsubsection{Corestriction}

Let $_gK=gKg^{-1}$. 
As before, the morphism in the chain complexes is given by induction of representations of the isotropy groups of each cell. We have the morphism at the level of Bredon chain complexes as follows:
$$\begin{tikzpicture}[scale=1, decoration={markings, mark=at position 1 with {\arrow{stealth}}}]
\node at (-2,2) {$0$};
\draw[postaction={decorate}] (-1.7,2) -- (-0.5,2);
\node at (0,2) {$\Z$};
\draw[postaction={decorate}] (0.5,2) -- (2.5,2);
\node at (3,2) {$\Z^{10}$};
\draw[postaction={decorate}] (3.5,2) -- (5.5,2);
\node at (6,2) {$\Z^{14}$};
\draw[postaction={decorate}] (6.5,2) -- (7.7,2);
\node at (8,2) {$0$};

\node at (-2,0) {$0$};
\draw[postaction={decorate}] (-1.7,0) -- (-0.5,0);
\node at (0,0) {$\Z^{3}$};
\draw[postaction={decorate}] (0.5,0) -- (2.5,0);
\node at (3,0) {$\Z^8$};
\draw[postaction={decorate}] (3.5,0) -- (5.5,0);
\node at (6,0) {$\Z^{14}$};
\draw[postaction={decorate}] (6.5,0) -- (7.7,0);
\node at (8,0) {$0$};

\node at (1.5,2.4) {$d^{\Gamma_1}_1$};
\node at (4.5,2.4) {$d^{\Gamma_1}_0$};
\node at (1.5,0.4) {$d^{{}_g K}_1$};
\node at (4.5,0.4) {$d^{{}_g K}_0$};

\draw[postaction={decorate}] (-0.1,0.4) -- (-0.1,1.7);
\node at (-0.4,1) {$g_2$};
\draw[postaction={decorate}] (2.9,0.4) -- (2.9,1.7);
\node at (2.6,1) {$g_1$};
\draw[postaction={decorate}] (5.9,0.4) -- (5.9,1.7);
\node at (5.6,1) {$g_0$};
\end{tikzpicture}$$
where the $g_i$ are the transposed matrices of the $f_i$.

In a similar way, the morphism 
$$\text{cores}: \mathcal{H}^{{}_gK}_{0}(X;\mathcal{R})\cong \Z^8\longrightarrow \mathcal{H}^{\Gamma_1}_{0}(X;\mathcal{R})\cong \Z^6$$
can be represented the matrix 
\begin{equation}\label{m2}\left(\begin{matrix}
1 & 0 & 0 & 0 & 0 & 0 & 0 & 0 \\
1 & 0 & 0 & 0 & 0 & 0 & 0 & 0 \\
0 & 1 & 1 & 1 & 0 & 0 & 0 & 0 \\
0 & 0 & 0 & 0 & 1 & 0 & 0 & 0 \\
0 & 0 & 0 & 0 & 0 & 1 & 1 & 0 \\
1 & 1 & 1 & 1 & 0 & -1 & 0 & 1
\end{matrix}\right).\end{equation}

\subsubsection{$Ad_g$}

For the conjugation morphism we have
$$\begin{tikzpicture}[scale=1, decoration={markings, mark=at position 1 with {\arrow{stealth}}}]
\node at (-2,2) {$0$};
\draw[postaction={decorate}] (-1.7,2) -- (-0.5,2);
\node at (0,2) {$\Z^{3}$};
\draw[postaction={decorate}] (0.5,2) -- (2.5,2);
\node at (3,2) {$\Z^{8}$};
\draw[postaction={decorate}] (3.5,2) -- (5.5,2);
\node at (6,2) {$\Z^{14}$};
\draw[postaction={decorate}] (6.5,2) -- (7.7,2);
\node at (8,2) {$0$};

\node at (-2,0) {$0$};
\draw[postaction={decorate}] (-1.7,0) -- (-0.5,0);
\node at (0,0) {$\Z^{3}$};
\draw[postaction={decorate}] (0.5,0) -- (2.5,0);
\node at (3,0) {$\Z^8$};
\draw[postaction={decorate}] (3.5,0) -- (5.5,0);
\node at (6,0) {$\Z^{14}$};
\draw[postaction={decorate}] (6.5,0) -- (7.7,0);
\node at (8,0) {$0$};

\node at (1.5,2.3) {$d^{K}_1$};
\node at (4.5,2.3) {$d^{K}_0$};
\node at (1.5,0.3) {$d^{{}_gK}_1$};
\node at (4.5,0.3) {$d^{{}_gK}_0$};

\draw[postaction={decorate}] (-0.1,1.7) -- (-0.1,0.4);
\node at (-0.4,1) {$h_2$};
\draw[postaction={decorate}] (2.9,1.7) -- (2.9,0.4);
\node at (2.6,1) {$h_1$};
\draw[postaction={decorate}] (5.9,1.7) -- (5.9,0.4);
\node at (5.6,1) {$h_0$};
\end{tikzpicture}$$
where the $h_i$ are the identity (the conjugation is not changing the order of the representations).

\subsubsection{Hecke operators}

Finally, we can compute the Hecke operators on the K-homology of the C*-algebra $C_r^*(\Gamma_1)$.

\begin{theorem}\label{main}
We have
$$K_n(C_r^*(\Gamma_1)) \cong \begin{cases}
\Z^6, & n=0; \\
\Z, & n=1.
\end{cases}$$
The Hecke operator 
\begin{eqnarray*}
T_{g,\pt} \, : \, K_n(C_r^*(\Gamma_1))
&\longrightarrow & K_n(C_r^*(\Gamma_1))
\end{eqnarray*}
for $n=0$ is given by the matrix 
$$\left(\begin{array}{cccccc}
1 & 2 & 0 & 0 & 0 & 0 \\
1 & 2 & 0 & 0 & 0 & 0 \\
0 & 0 & 3 & 0 & 0 & 0 \\
0 & 1 & 1 & 1 & 0 & 0 \\
1 & 1 & 2 &-1 & 2 &-1 \\
0 & 1 & 1 & 1 & -1 & 2
\end{array}\right)$$ and is zero for $n=1$.
\end{theorem}
\begin{proof}
First note that as $\Gamma_1$ satisfies the Baum-Connes conjecture, the discussion at the end of Section \ref{K-homology} implies $$K_0(C_r^*(\Gamma_1))\cong \mathcal{H}_0^{\Gamma_1}(\underbar{E}\Gamma_1;\mathcal{R})\cong\Z^6 \quad\text{and}\quad K_1(C_r^*(\Gamma_1))\cong \mathcal{H}_1^{\Gamma_1}(\underbar{E}\Gamma_1;\mathcal{R})\cong\Z.$$ 
Then, Theorems \ref{comm} and \ref{bredon} imply that $T_{g,\pt}$ can be represented the same as the Hecke operator on Bredon homology. For $n=0$, it is the product of matrices \ref{m1} and \ref{m2}; for $n=1$ it is trivial, since it factors through $\mathcal{H}_1^{K}(\underbar{E}K;\mathcal{R})=0$.
\end{proof}

\section{Concluding remarks}
These computations are a first step towards developing a general algorithm for working with Hecke operators in Bredon cohomology. The implementation of such algorithm in GAP is part of forthcoming work. As mentioned above, an initial GAP algorithm to compute $\uE\Gamma_1/K$ for subgroups $K$ associated to other primes $p$ in $\Z[i]$ can be found in \cite{Munoz20}.

Relations of our work with arithmetic aspects of the theory involving automorphic forms are still to be studied. We expect that using the  above set up the study of such relations will lead to fruitful developments of the theory.


\bibliographystyle{alpha}
\bibliography{BianchiHeckeBib}

\end{document}